\documentclass[12pt, reqno]{amsart}

\usepackage[T1]{fontenc}

\usepackage{amsmath, amssymb, amsfonts, amsthm}
\usepackage{mathtools}
\usepackage{latexsym}

\usepackage{amsbsy}

\usepackage[left=30mm, right=30mm, top=20mm, bottom=20mm]{geometry}

\usepackage{enumitem}
\usepackage{comment}

\usepackage{xcolor}

\usepackage{tikz}
\usetikzlibrary{circuits, intersections}
\usetikzlibrary{plotmarks}
\usetikzlibrary{angles, quotes}

\usepackage{graphicx}
\usepackage{pdfpages}

\usepackage[super]{nth}

\usepackage{hyperref}
\usepackage[colorinlistoftodos]{todonotes}

\newtheorem{thmintro}{Theorem}

\newtheorem{corintro}[thmintro]{Corollary}

\theoremstyle{definition}
\newtheorem{rkintro}{Remark}

\theoremstyle{plain}
\newtheorem{theorem}{Theorem}[section]
\newtheorem{proposition}[theorem]{Proposition}
\newtheorem{lemma}[theorem]{Lemma}
\newtheorem{corollary}[theorem]{Corollary}

\theoremstyle{definition}

\newtheorem{definition}[theorem]{Definition}
\newtheorem{convention}[theorem]{Convention}
\newtheorem{example}[theorem]{Example}

\theoremstyle{remark}
\newtheorem{remark}[theorem]{Remark}

\renewcommand{\phi}{\varphi}
\renewcommand{\epsilon}{\varepsilon}

\newcommand{\NN}{\mathbb{N}}
\newcommand{\ZZ}{\mathbb{Z}}

\newcommand{\FF}{\mathbb{F}}

\newcommand{\KK}{\mathbb{K}}

\newcommand{\C}{\mathcal{C}}
\renewcommand{\P}{\mathcal{P}}
\newcommand{\Cbf}{\textbf{\textup{C}}}

\DeclareMathOperator{\id}{id}

\DeclareMathOperator{\Aut}{Aut}
\DeclareMathOperator{\Sym}{Sym}

\numberwithin{equation}{section} 
\begin{document}

\title{RGD-systems over $\FF_2$}
\author{Sebastian Bischof}

\thanks{email: bischof.math@icloud.com}

\thanks{Mathematisches Institut, Arndtstra\ss e 2, 35392 Gie\ss en, Germany}



\begin{abstract}
	In this paper we prove that an RGD-system over $\FF_2$ with prescribed commutation relations exists if and only if the commutation relations are Weyl-invariant and can be realized in the group $U_+$. This result gives us a machinery to produce new examples of RGD-systems with complicated commutation relations. We also discuss some applications of this result.
\end{abstract}

\maketitle

\section{Introduction}

In \cite{Ti92} Tits introduced RGD-systems in order to describe groups of Kac-Moody type. By definition, every RGD-system has a \emph{type} which is given by a Coxeter system, and to every Coxeter system one can associate its set of roots $\Phi$ (viewed as half spaces). An \emph{RGD-system of type $(W, S)$} is a pair $\left( G, (U_{\alpha})_{\alpha \in \Phi} \right)$ consisting of a group $G$ together with a family of subgroups $(U_{\alpha})_{\alpha \in \Phi}$ called \emph{root subgroups} indexed by the set of roots $\Phi$ satisfying some axioms. One key axiom makes an assumption about the commutation relations between root groups corresponding to prenilpotent pairs of roots $\{ \alpha, \beta \}$ where $(\alpha, \beta)$ is a finite set of roots determined by $\alpha$ and $\beta$ (we refer to Section \ref{sec:Premilinaries} for the precise definitions):
\begin{equation}
	[U_{\alpha}, U_{\beta}] \leq \langle U_{\gamma} \mid \gamma \in (\alpha, \beta) \rangle. \tag{RGD$1$}
\end{equation}

A fundamental question is whether we can determine all possible commutation relations. For each root $\alpha$ we denote by $r_{\alpha}$ the unique reflection which interchanges $\alpha$ and its opposite root. Let $\{ \alpha, \beta \}$ be a prenilpotent pair of roots. If $o(r_{\alpha} r_{\beta}) < \infty$, then the commutation relation between $U_{\alpha}$ and $U_{\beta}$ is \emph{known} from the classification of Moufang polygons by Tits and Weiss \cite{TW02}. They have shown that there exists a parametrization of the root groups by some algebraic structures and the commutation relations between these root groups can be expressed in terms of this parametrization. In the case $o(r_{\alpha} r_{\beta}) = \infty$ the situation is not so well understood yet. According to current knowledge it is unknown how complicated the commutation relations can be. We will later in the introduction come back to this question.

Originally it was our motivation to construct new examples of RGD-systems of \emph{$2$-spherical} type with non-trivial commutation relations. Due to results of M\"uhlherr-Ronan \cite{MR95} and Abramenko-M\"uhlherr \cite{AM97}, such a construction is only feasible if the root groups are not too large. But in this case the action of the torus $H := \bigcap_{\alpha \in \Phi} N_G(U_{\alpha})$ -- in general a useful tool, which also restricts the commutation relations -- does not have such a deep impact. Even worse, if all root groups have cardinality $2$ (we call such RGD-systems \emph{over} $\FF_2$), then the torus is trivial. On the other hand, a trivial torus simplifies the construction of RGD-systems, which is in general a difficult problem.

We are interested in conditions/restrictions on the commutation relations which ensure the existence of an RGD-system with these prescribed commutation relations. As already mentioned, the commutation relations between prenilpotent pairs of roots $\{\alpha, \beta\}$ with $o(r_{\alpha} r_{\beta}) < \infty$ are prescribed by \cite{TW02}, but we have some flexibility in the case $o(r_{\alpha} r_{\beta}) = \infty$. 

In this paper we investigate RGD-systems over $\FF_2$. In Section \ref{Section: commutator blueprints} we introduce the notion of \emph{commutator blueprints}, which are purely combinatorial objects. They prescribe the structure of commutation relations between prenilpotent pairs of \emph{positive} roots and give rise to the groups $U_w$ -- these groups appear naturally as subgroups of RGD-systems and are generated by suitable root groups. We denote the direct limit of the groups $U_w$ by $U_+$. To each RGD-system over $\FF_2$ one can associate a commutator blueprint. Such blueprints are called \emph{integrable}. One can show that every integrable commutator blueprint is \emph{faithful} (the canonical homomorphisms $U_w \to U_+$ are injective) and \emph{Weyl-invariant} (roughly speaking:\ the commutation relations are Weyl-invariant). It turns out that these two necessary conditions of integrability are already sufficient and lead to the main result of this article (cf.\ Remark \ref{Remark: Weyl-invariance} and Theorem \ref{RGDsystem}):

\begin{thmintro}\label{Main result: Theorem}
	For every commutator blueprint $\mathcal{M}$, the following are equivalent:
	\begin{enumerate}[label=(\roman*)]
		\item $\mathcal{M}$ is integrable.
		
		\item $\mathcal{M}$ is faithful and Weyl-invariant.
	\end{enumerate}
\end{thmintro}

\begin{rkintro}
	In a sense, Theorem \ref{Main result: Theorem} confirms a quote from Tits in the special case of $\FF_2$, where he claims that constructing the Borel subgroup (here:\ $U_+$) is as hard as constructing the whole RGD-system (cf.\ \cite[Ch.\ 3.4 in \emph{Buildings and group amalgamations}]{Ti13}).
\end{rkintro}

\begin{rkintro}
	The question whether an RGD-system with prescribed commutation relations exists reduces by Theorem \ref{Main result: Theorem} to the question of existence of a suitable faithful and Weyl-invariant commutator blueprint. In \cite{BiConstruction} we discussed the existence of commutator blueprints in more detail. In particular, we constructed examples of commutator blueprints and in all constructions the Weyl-invariance followed directly from the description of the commutator blueprint. We constructed the groups $U_w$ as semi-direct products $U_{w'} \rtimes \ZZ_2$ and we have worked out explicit conditions on the commutation relations ensuring they define a commutator blueprint. In general, it is hard to decide whether a given commutator blueprint if faithful. We mention here two special classes. 
	
	If $(W, S)$ is of \emph{universal} type (i.e.\ $o(st) = \infty$ for all $s\neq t \in S$), then the group $U_+$ is the tree product of the groups $U_w$. Using \cite[Ch.\ $4.4$]{Se79}, the homomorphisms $U_w \to U_+$ are injective and, in particular, every commutator blueprint of universal type is automatically faithful. Thus our main result reduces the difficult problem about the existence of RGD-systems to the existence of the finite groups $U_w$.
	
	On the other hand, if $(W, S)$ is \emph{of type} $(4, 4, 4)$, that is, $(W, S)$ is of rank $3$ and $o(st) = 4$ for all $s\neq t \in S$, then we have shown in \cite{BiDiss} that every Weyl-invariant commutator blueprint of type $(W, S)$ is faithful. Together with Theorem \ref{Main result: Theorem} this allows us to construct new examples of RGD-systems of $2$-spherical type.
\end{rkintro}

\subsection*{Commutation relations}

We now come back to our earlier question of determining how complicated the commutation relations between $U_{\alpha}$ and $U_{\beta}$ in the case $o(r_{\alpha} r_{\beta}) = \infty$ can be. We will see that in all known examples the commutation relations are "simple".

\begin{enumerate}[label=(\alph*)]
	\item Kac-Moody groups $\mathbf{G}$ form the most natural family of examples in this context, as the axioms of RGD-systems are motivated by the theory of Kac-Moody groups. For a field $\FF$, the root groups of $\mathbf{G}(\FF)$ are parametrized by $(\FF, +)$ and we have explicit commutation relations for prenilpotent pairs of roots. Moreover, one can show that $[U_{\alpha}, U_{\beta}] \leq U_{\alpha + \beta}$ holds (cf.\ \cite[Proposition $1$]{BP95} and \cite[Theorem $2$]{Mo87}).
	
	\item In the $2$-spherical case it is known that the commutation relation between $U_{\alpha}$ and $U_{\beta}$ is uniquely determined by the \emph{local} commutation relations, if the root groups are large enough (cf.\ \cite{AM97}). Moreover, it turned out that in the rank $3$ case they are almost always trivial (cf.\ \cite{Bi22}). This result does also hold in the \emph{simply-laced} case or if $o(st) \geq 3$ for all $s\neq t\in S$. In the \emph{right-angled} case, the situation is more complicated.
	
	\item\label{Tits} In \cite[Section $5.4$ in $95/96$]{Ti73-00} Tits has constructed uncountably many isomorphism classes of trivalent \emph{Moufang twin trees} which are essentially the same as (center-free) RGD-systems of type $\tilde{A}_1$ over $\FF_2$. In the case $\tilde{A}_1$ we have $\Phi = \{+, -\} \times \ZZ$ and all the examples constructed by Tits are of the following form, where $\epsilon \in \{+, -\}$: 
	\allowdisplaybreaks
	\begin{align*}
		&&\forall z, z' \in \ZZ: &&[U_{{\epsilon, 2z}}, U_{\epsilon, z'}] = 1 &&\text{and} &&[U_{\epsilon, 2z+1}, U_{\epsilon, 2z'+1}] \leq \langle U_{\epsilon, 2i} \mid i \in \ZZ \rangle.
	\end{align*}
	Gr\"uninger, Horn and M\"uhlherr announced in \cite{GHM16} the existence of an RGD-system of type $\tilde{A}_1$ over $\FF_2$, which has different commutation relations than those constructed by Tits. Unfortunately, for both results there are no proofs available in the literature yet. However, the existence of the trivalent Moufang twin trees constructed by Tits follows from Theorem \ref{Main result: Theorem} together with \cite[Theorem A]{BiConstruction}. Moreover, we construct independently an RGD-system of type $\tilde{A}_1$ over $\FF_2$ which has different commutation relations than those constructed by Tits. The existence follows from Theorem \ref{Main result: Theorem} and \cite[Theorem $4.8$]{BiConstruction}. Thus Theorem \ref{Main result: Theorem} provides a new approach to these already known results.
	
	\item\label{GHM} Gr\"uninger, Horn and M\"uhlherr have shown in \cite{GHM16} that in the case $\tilde{A}_1$ the commutation relations are generally very restrictive. It is a consequence of \cite[Theorem A]{GHM16} that for all $w\in W$ the group $U_w$ is nilpotent of class at most $2$, provided that all root groups are isomorphic to $(\FF_p, +)$ for a fixed prime $p$. This result was generalized in \cite{PaDiss21} and includes the cases where $U_{\alpha_s} \cong (\KK_s, +)$ with $\KK_s$ a field of characteristic different from $2$ (cf.\ also \cite{SW08}, \cite{Seg09}).
	
	\item\label{Construction Ronan-Remy} In \cite{RR06} Rémy and Ronan have constructed \emph{exotic} RGD-systems of right-angled type $(W, S)$ with prescribed isomorphism types of the root groups. More precisely, they have established the existence of RGD-systems of type $(W, S)$ with $U_{\alpha_s} \cong (\KK_s, +)$ for every prescribed family of fields $(\KK_s)_{s\in S}$ -- and hence the existence of RGD-systems of mixed characteristics of root groups. One main aspect in their construction is that root groups corresponding to prenilpotent pairs of roots commute.
	
	\item Let $(G, (U_{\alpha})_{\alpha \in \Phi})$ be an RGD-system of type $(W, S)$ in which every root group is of \emph{prime exponent}, that is, for each $\alpha \in \Phi$ there exists a prime $p_{\alpha}$ with $g^{p_{\alpha}} = 1$ for all $g\in U_{\alpha}$ (e.g.\ any example constructed in \cite{RR06}, where the fields are of positive characteristic). Then one can show
	\[ [U_{\alpha}, U_{\beta}] \leq \langle U_{\gamma} \mid \gamma \in (\alpha, \beta), \, p_{\alpha} = p_{\gamma} = p_{\beta} \rangle. \]
	This leads to strong restrictions of the commutation relations in the case of different prime exponents (e.g.\ $p_{\alpha} \neq p_{\beta}$ implies $[U_{\alpha}, U_{\beta}] = 1$). This explains also the choice of the commutation relations in \ref{Construction Ronan-Remy} for fields having different characteristic. On the other hand, if all root groups have the same prime exponent, then the previous inclusion is just the original axiom (RGD$1$). In particular, this obstruction does not occur in RGD-systems over $\FF_2$ and allows us to construct RGD-systems with complicated commutation relations.
\end{enumerate}

\subsection*{Consequences}

In the rest of the introduction we will discuss consequences of our main result. We first discuss property (FPRS) of an RGD-system, introduced by Caprace and Rémy in \cite[$2.1$]{CR09}; we refer to loc.\ cit.\ for more information. This property makes a statement about the set of fixed points of the action of the root groups on the associated building. It implies that every root group is contained in a suitable contraction group. Property (FPRS) is used in \cite{CR09} to show that -- under some mild conditions -- the \emph{geometric completion} of an RGD-system (cf.\ \cite{RR06}) is topologically simple. Caprace and Rémy have shown in \cite{CR09} that almost all RGD-systems of $2$-spherical type with finite root groups, the exotic examples in \cite{RR06} as well as Kac-Moody groups satisfy property (FPRS).

According to \cite{CR09} it has been known that there exist RGD-systems that do not satisfy (FPRS) and we refer to Remark \ref{Remark: FPRS} for more information. The following corollary provides the existence of such RGD-systems in each rank (Corollary \ref{Corollary: 2^N examples}):

\begin{corintro}
	For each universal Coxeter system $(W, S)$ of rank at least $2$ there exists an RGD-system of type $(W, S)$ over $\FF_2$ which does not satisfy property (FPRS).
\end{corintro}

In \cite{CRW17b} Caprace, Reid and Willis initiated a systematic study of the class $\mathcal{S}$ consisting of non-discrete, compactly generated, topologically simple, totally disconnected, locally compact groups. As we have mentioned before, property (FPRS) implies that the geometric completion of an RGD-system with finite root groups belongs to the class $\mathcal{S}$. In general, it is a difficult problem to construct new (families of) examples of groups in $\mathcal{S}$. However, if there exists a constant $C \geq 0$ such that for all prenilpotent pairs $\{ \alpha, \beta \}$ of roots with $[U_{\alpha}, U_{\beta}] \neq 1$ the distance between the corresponding walls is at most $C$, it follows similarly as in \cite[Lemma $5$]{CR09} that property (FPRS) is satisfied. Using Theorem \ref{Main result: Theorem}, we can then produce many new examples of groups in $\mathcal{S}$.

The second application concerns the nilpotency class of the groups $U_w$ in RGD-systems ($U_w$ is nilpotent if the root groups are nilpotent). The nilpotency class of the groups $U_w$ in the case of type $\tilde{A}_1$ is at most $2$ (cf.\ \ref{GHM} above). Thanks to a result of Gl\"ockner and Willis about contraction groups, this result generalizes to all types as follows: If (FPRS) is satisfied and if all root groups have cardinality $p$ for a fixed prime $p$, then it follows from \cite[Theorem A]{GW21} that the nilpotency class of the group $U_{w^k}$ for $k \in \NN$ is bounded above by $\ell(w)$, if the element $w$ is \emph{straight} (i.e.\ if $\ell(w^k) = \vert k \vert \ell(w)$ for all $k \in \ZZ$).

In \cite[Theorem $1.2$]{Ca07} Caprace has proved that the nilpotency class of the groups $U_w$ in Kac-Moody groups of arbitrary type is bounded above by a constant only depending on the generalized Cartan matrix $A$ and not on $w$. We will see that the general situation is very different and the results about Kac-Moody groups do not generalize to arbitrary RGD-systems. Even more, we can construct for each $m\geq 3$ an example of an RGD-system of rank $m$ such that the nilpotency class of the groups $U_w$ can be arbitrarily large. To make the statement precise, for an RGD-system $\mathcal{D}$ we define $\mathrm{ndeg}(\mathcal{D})$ to be the supremum of the nilpotency classes of the subgroups $U_w$ for all $w\in W$. The following result follows from Theorem \ref{Main result: Theorem} together with \cite[Theorem B$\&$C]{BiConstruction}.

\begin{corintro}\label{Main result: Cor nilpotency class n}
	Let $(W, S)$ be a universal Coxeter system of rank $m \geq 3$.
	\begin{enumerate}[label=(\alph*)]
		\item For each $n \in \NN$ there is an RGD-system $\mathcal{D}_n$ of type $(W, S)$ with $\mathrm{ndeg}(\mathcal{D}_n) = n$.
		
		\item There exists an RGD-system $\mathcal{D}$ of type $(W, S)$ with $\mathrm{ndeg}(\mathcal{D}) = \infty$.
	\end{enumerate}
\end{corintro}

\subsection*{Overview}

We sketch here the proof strategy of Theorem \ref{Main result: Theorem} to help the reader get a rough overview of the structure of this article. Before we start let us mention that Section \ref{sec:Premilinaries} is devoted to fixing notation and in Section \ref{Section: commutator blueprints} we introduce the notion of commutator blueprints and prove elementary facts about them.

Let $\mathcal{M}$ be a faithful and Weyl-invariant commutator blueprint of type $(W, S)$ and let $s\in S$. We will construct in Section \ref{Section: Constructing the rank 1 parabolics} the \emph{rank $1$ parabolics} $P_s$ as follows: We first observe that we can decompose the group $U_+$ into a semi-direct product $U_+ \cong U_s \ltimes N_s$. Proposition \ref{Proposition: Existence of taus} and Corollary \ref{ustaus} then imply the existence of an automorphism $\tau_s \in \Aut(N_s)$ with $\tau_s(U_{\alpha}) = U_{s\alpha}$. Moreover, we have $\tau_s^2 = 1 = (u_s \tau_s)^3$ in $\Aut(N_s)$ where $u_s \in U_s$ denotes the non-trivial element. We define $P_s := \Sym(3) \ltimes N_s$, where $\Sym(3) = \langle u_s, \tau_s \rangle$. Next, we will construct an RGD-system containing the groups $P_s$ as subgroups. In order to do that, we introduce in Section \ref{Section: Action of Ps} a chamber system $\Cbf$ and show that all $P_s$ act on $\Cbf$. It turns out that this action is faithful (cf.\ Proposition \ref{Proposition: Ps acts on C}) and that the braid relations $(\tau_s \tau_t)^{o(st)}$ act trivially on $\Cbf$ (Theorem \ref{Theorem: braid relations act trivial}). We define $G$ to be the direct limit of the inductive system formed by the groups $U_+, (P_s)_{s\in S}, (\langle \tau_s \rangle)_{s\in S}, W \cong \langle \tau_s \mid s\in S \rangle$. Note that $G$ acts non-trivially on the chamber system $\Cbf$. There is a canonical way of defining root groups $U_{\alpha}$ inside $G$ for all $\alpha \in \Phi$ as conjugates of the root groups corresponding to simple roots and we put $\mathcal{D}_{\mathcal{M}} := (G, (U_{\alpha})_{\alpha \in \Phi})$. Using the action of $G$ on $\Cbf$ we observe in Theorem \ref{RGDsystem} that $\mathcal{D}_{\mathcal{M}}$ is an RGD-system and $\mathcal{M}$ is integrable.

\subsection*{Acknowledgement}

I am very grateful to Bernhard M\"uhlherr for proposing this project to me. Moreover, I would like to thank him for many helpful discussions on the topic as well as for valuable comments which improved the presentation of the present results. I also thank Timothée Marquis, Colin Reid and George Willis for valuable remarks on an earlier draft.

\section{Preliminaries}\label{sec:Premilinaries}

\subsection*{Coxeter systems}

Let $(W, S)$ be a Coxeter system and let $\ell$ denote the corresponding length function. For $s, t \in S$ we denote the order of $st$ in $W$ by $m_{st}$. The \emph{Coxeter diagram} corresponding to $(W, S)$ is the labeled graph $(S, E(S))$, where $E(S) = \{ \{s, t \} \mid m_{st}>2 \}$ and where each edge $\{s,t\}$ is labeled by $m_{st}$ for all $s, t \in S$. The \emph{rank} of the Coxeter system is the cardinality of the set $S$.

It is well-known that for each $J \subseteq S$ the pair $(\langle J \rangle, J)$ is a Coxeter system (cf.\ \cite[Ch. IV, §$1$ Theorem $2$]{Bo68}). A subset $J \subseteq S$ is called \emph{spherical} if $\langle J \rangle$ is finite. The Coxeter system is called \emph{spherical} if $S$ is spherical. Given a spherical subset $J$ of $S$, there exists a unique element of maximal length in $\langle J \rangle$, which we denote by $r_J$ (cf.\ \cite[Corollary $2.19$]{AB08}).

\begin{lemma}\label{conditionF}
	Let $\epsilon \in \{+, -\}$ and let $(W, S)$ be a Coxeter system. Suppose $s, t \in S$ and $w\in W$ with $\ell(sw) = \ell(w) \epsilon 1 = \ell(wt)$. Then either $\ell(swt) = \ell(w) \epsilon 2$ or else $swt = w$.
\end{lemma}
\begin{proof}
	The case $\epsilon = +$ is \cite[Condition ($\mathbf{F}$) on p. $79$]{AB08}. The case $\epsilon = -$ can be deduced from the case $\epsilon = +$.
\end{proof}

\subsection*{Buildings}

Let $(W, S)$ be a Coxeter system. A \emph{building of type $(W, S)$} is a pair $\Delta = (\C, \delta)$ where $\C$ is a non-empty set and where $\delta: \C \times \C \to W$ is a \emph{distance function} satisfying the following axioms, where $x, y\in \C$ and $w = \delta(x, y)$:
\begin{enumerate}[label=(Bu\arabic*)]
	\item $w = 1_W$ if and only if $x=y$;
	
	\item if $z\in \C$ satisfies $s := \delta(y, z) \in S$, then $\delta(x, z) \in \{w, ws\}$, and if, furthermore, $\ell(ws) = \ell(w) +1$, then $\delta(x, z) = ws$;
	
	\item if $s\in S$, there exists $z\in \C$ such that $\delta(y, z) = s$ and $\delta(x, z) = ws$.
\end{enumerate}
The \emph{rank} of $\Delta$ is the rank of the underlying Coxeter system. The elements of $\C$ are called \emph{chambers}. Given $s\in S$ and $x, y \in \C$, then $x$ is called \emph{$s$-adjacent} to $y$, if $\delta(x, y) = s$. The chambers $x, y$ are called \emph{adjacent}, if they are $s$-adjacent for some $s\in S$. A \emph{gallery} from $x$ to $y$ is a sequence $(x = x_0, \ldots, x_k = y)$ such that $x_{l-1}$ and $x_l$ are adjacent for all $1 \leq l \leq k$; the number $k$ is called the \emph{length} of the gallery. Let $(x_0, \ldots, x_k)$ be a gallery and suppose $s_i \in S$ with $\delta(x_{i-1}, x_i) = s_i$. Then $(s_1, \ldots, s_k)$ is called the \emph{type} of the gallery. A gallery from $x$ to $y$ of length $k$ is called \emph{minimal} if there is no gallery from $x$ to $y$ of length $<k$.

Given a subset $J \subseteq S$ and $x\in \C$, the \emph{$J$-residue} of $x$ is the set $R_J(x) := \{y \in \C \mid \delta(x, y) \in \langle J \rangle \}$. Each $J$-residue is a building of type $(\langle J \rangle, J)$ with the distance function induced by $\delta$ (cf.\ \cite[Corollary $5.30$]{AB08}). A \emph{residue} is a subset $R$ of $\C$ such that there exist $J \subseteq S$ and $x\in \C$ with $R = R_J(x)$. Since the subset $J$ is uniquely determined by $R$, the set $J$ is called the \emph{type} of $R$ and the \emph{rank} of $R$ is defined to be the cardinality of $J$. A residue is called \emph{spherical} if its type is a spherical subset of $S$. Let $R$ be a spherical $J$-residue. Then $x, y \in R$ are called \emph{opposite in $R$} if $\delta(x, y) = r_J$. A \emph{panel} is a residue of rank $1$. The building $\Delta$ is called \emph{thick}, if each panel of $\Delta$ contains at least three chambers. A building is called \emph{spherical} if its type is spherical.

An \emph{(type-preserving) automorphism} of a building $\Delta = (\C, \delta)$ is a bijection $\phi:\C \to \C$ such that $\delta(\phi(c), \phi(d)) = \delta(c, d)$ holds for all chambers $c, d \in \C$. We remark that some authors distinguish between automorphisms and type-preserving automorphisms. An automorphism in our sense is type-preserving. We denote the set of all automorphisms of the building $\Delta$ by $\Aut(\Delta)$.

\begin{example}
	We define $\delta: W \times W \to W, (x, y) \mapsto x^{-1}y$. Then $\Sigma(W, S) := (W, \delta)$ is a building of type $(W, S)$. The group $W$ acts faithful on $\Sigma(W, S)$ by multiplication from the left, i.e.\ $W \leq \Aut(\Sigma(W, S))$.
\end{example}

\begin{theorem}\label{Theorem5.205AB08}
	Let $\Delta = (\C, \delta)$ be a thick spherical building of type $(W, S)$ and let $c, d \in \C$ be opposite chambers in $\C$. Then the only automorphism of $\Delta$, which fixes $\bigcup_{s\in S} R_{\{s\}}(c) \cup \{d\}$ pointwise, is the identity.
\end{theorem}
\begin{proof}
	This is \cite[Theorem $5.205$]{AB08}.
\end{proof}

\subsection*{Roots}

Let $(W, S)$ be a Coxeter system. A \emph{reflection} is an element of $W$ that is conjugate to an element of $S$. For $s\in S$ we let $\alpha_s := \{ w\in W \mid \ell(sw) > \ell(w) \}$ be the \emph{simple root} corresponding to $s$. A \emph{root} is a subset $\alpha \subseteq W$ such that $\alpha = v\alpha_s$ for some $v\in W$ and $s\in S$. We denote the set of all roots by $\Phi(W, S)$. The set $\Phi(W, S)_+ := \{ \alpha \in \Phi(W, S) \mid 1_W \in \alpha \}$ is the set of all \emph{positive roots} and $\Phi(W, S)_- := \{ \alpha \in \Phi(W, S) \mid 1_W \notin \alpha \}$ is the set of all \emph{negative roots}. For each root $\alpha \in \Phi(W, S)$ we denote its \emph{opposite} root by $-\alpha$ and we denote the unique reflection which interchanges these two roots by $r_{\alpha} \in W \leq \Aut(\Sigma(W, S))$. A pair $\{ \alpha, \beta \}$ of roots is called \emph{prenilpotent} if both $\alpha \cap \beta$ and $(-\alpha) \cap (-\beta)$ are non-empty sets. For such a pair we will write $\left[ \alpha, \beta \right] := \{ \gamma \in \Phi(W, S) \mid \alpha \cap \beta \subseteq \gamma \text{ and } (-\alpha) \cap (-\beta) \subseteq -\gamma \}$ and $(\alpha, \beta) := \left[ \alpha, \beta \right] \backslash \{ \alpha, \beta \}$.

\begin{convention}
	For the rest of this paper we let $(W, S)$ be a Coxeter system of finite rank and we define $\Phi := \Phi(W, S)$ (resp.\ $\Phi_+, \Phi_-$). 
\end{convention}

\subsection*{Coxeter buildings}

In this subsection we consider the Coxeter building $\Sigma(W, S)$. For $\alpha \in \Phi$ we denote by $\partial \alpha$ (resp.\ $\partial^2 \alpha$) the set of all panels (resp.\ spherical residues of rank $2$) stabilized by $r_{\alpha}$. The set $\partial \alpha$ is called the \emph{wall} associated with $\alpha$. Let $G = (c_0, \ldots, c_k)$ be a gallery. We say that $G$ \emph{crosses the wall $\partial \alpha$} if there exists $1 \leq i \leq k$ such that $\{ c_{i-1}, c_i \} \in \partial \alpha$. It is a basic fact that a minimal gallery crosses a wall at most once (cf.\ \cite[Lemma $3.69$]{AB08}). Let $(c_0, \ldots, c_k)$ and $(d_0 = c_0, \ldots, d_k = c_k)$ be two minimal galleries from $c_0$ to $c_k$ and let $\alpha \in \Phi$. Then $\partial \alpha$ is crossed by the minimal gallery $(c_0, \ldots, c_k)$ if and only if it is crossed by the minimal gallery $(d_0, \ldots, d_k)$. Moreover, a gallery which crosses each wall at most once is already minimal. For $\alpha_1, \ldots, \alpha_k \in \Phi$ we say that a minimal gallery $G = (c_0, \ldots, c_k)$ \emph{crosses the sequence of roots} $(\alpha_1, \ldots, \alpha_k)$, if $c_{i-1} \in \alpha_i$ and $c_i \notin \alpha_i$ for all $1\leq i \leq k$.

We denote the set of all minimal galleries $(c_0 = 1_W, \ldots, c_k)$ starting at $1_W$ by $\mathrm{Min}$. For $w\in W$ we denote the set of all $G \in \mathrm{Min}$ of type $(s_1, \ldots, s_k)$ with $w = s_1 \cdots s_k$ by $\mathrm{Min}(w)$. For $w\in W$ with $\ell(sw) = \ell(w) -1$ we let $\mathrm{Min}_s(w)$ be the set of all $G \in \mathrm{Min}(w)$ of type $(s, s_2, \ldots, s_k)$. We extend this notion to the case $\ell(sw) = \ell(w) +1$ by defining $\mathrm{Min}_s(w) := \mathrm{Min}(w)$. Let $w\in W, s\in S$ and $G = (c_0, \ldots, c_k) \in \mathrm{Min}_s(w)$. If $\ell(sw) = \ell(w) -1$, then $c_1 = s$ and we define $sG := (sc_1 = 1_W, \ldots, sc_k) \in \mathrm{Min}(sw)$. If $\ell(sw) = \ell(w) +1$, we define $sG := (1_W, sc_0 = s, \ldots, sc_k) \in \mathrm{Min}(sw)$.

\subsection*{Root group data}

An \emph{RGD-system of type $(W, S)$} is a pair $\mathcal{D} = \left( G, \left( U_{\alpha} \right)_{\alpha \in \Phi}\right)$ consisting of a group $G$ together with a family of subgroups $U_{\alpha}$ (called \emph{root groups}) indexed by the set of roots $\Phi$, which satisfies the following axioms, where $H := \bigcap_{\alpha \in \Phi} N_G(U_{\alpha})$ and $U_{\epsilon} := \langle U_{\alpha} \mid \alpha \in \Phi_{\epsilon} \rangle$ for $\epsilon \in \{+, -\}$:
\begin{enumerate}[label=(RGD\arabic*)] \setcounter{enumi}{-1}
	\item For each $\alpha \in \Phi$, we have $U_{\alpha} \neq \{1\}$.
	
	\item For each prenilpotent pair $\{ \alpha, \beta \} \subseteq \Phi$ with $\alpha \neq \beta$, the commutator group $[U_{\alpha}, U_{\beta}]$ is contained in the group $U_{(\alpha, \beta)} := \langle U_{\gamma} \mid \gamma \in (\alpha, \beta) \rangle$.
	
	\item For each $s\in S$ and each $u\in U_{\alpha_s} \backslash \{1\}$, there exist $u', u'' \in U_{-\alpha_s}$ such that the product $m(u) := u' u u''$ conjugates $U_{\beta}$ onto $U_{s\beta}$ for each $\beta \in \Phi$.
	
	\item For each $s\in S$, the group $U_{-\alpha_s}$ is not contained in $U_+$.
	
	\item $G = H \langle U_{\alpha} \mid \alpha \in \Phi \rangle$.
\end{enumerate}
For $w\in W$ we define $U_w := \langle U_{\alpha} \mid w \notin \alpha \in \Phi_+ \rangle$. Let $G \in \mathrm{Min}(w)$ and let $(\alpha_1, \ldots, \alpha_k)$ be the sequence of roots crossed by $G$. Then we have $U_w = U_{\alpha_1} \cdots U_{\alpha_k}$. Following \cite[Remark $(1)$ on p. $258$]{Ti92} we have $m_{st} \in \{2, 3, 4, 6, 8, \infty \}$ for all $s \neq t \in S$. An RGD-system $\mathcal{D} = (G, (U_{\alpha})_{\alpha \in \Phi})$ is said to be \emph{over $\FF_2$} if every root group has cardinality $2$.

\begin{example}\label{exprank2rgd}
	Let $(W, S)$ be spherical and of rank $2$ and let $\mathcal{D} = (G, (U_{\alpha})_{\alpha \in \Phi})$ be an RGD-system of type $(W, S)$ over $\FF_2$. For $S = \{s, t\}$ we deduce $m_{st} \in \{2, 3, 4, 6 \}$, since in an octagon there exists a root group of cardinality at least $4$ (cf.\ \cite[$16.9$ and $17.7$]{TW02}). Let $G \in \mathrm{Min}(r_S)$ and let $(\beta_1, \ldots, \beta_m)$ be the sequence of roots crossed by $G$, where $m = m_{st}$. Then $\Phi_+ = \{ \beta_1, \ldots, \beta_m \}$ and $\beta_1, \beta_m$ are the two simple roots. We let $U_{\beta_i} = \langle u_i \rangle$. For all $1 \leq i < j \leq m$ we will define subsets $M_{\{\beta_i, \beta_j\}} \subseteq (\beta_i, \beta_j)$ which correspond to the commutation relations. If $[u_i, u_j] = 1$, we put $M_{\{\beta_i, \beta_j\}} := \emptyset$. We now state all non-trivial commutation relations depending on the type $(W, S)$ (cf.\ \cite[Ch. $16, 17$]{TW02}):
	\begin{enumerate}
		\item[$A_1 \times A_1$:] There are no non-trivial commutation relations.
		
		\item[$A_2$:] There is only one non-trivial commutation relation, namely $[u_1, u_3] = u_2$ (cf.\ \cite[$16.1, 17.2$]{TW02}). We define $M_{\{\beta_1, \beta_3\}} = \{ \beta_2 \}$.
		
		\item[$B_2 = C_2$:] As in the case of $A_2$ there is only one non-trivial commutation relation, namely $[u_1, u_4] = u_2 u_3$ (cf.\ \cite[$16.2, 17.4$]{TW02} and \cite[$5.2.3$]{PT84}). We define $M_{\{\beta_1, \beta_4\}} := \{ \beta_2, \beta_3 \}$.
		
		\item[$G_2$:] We have the following non-trivial commutation relations (cf.\ \cite[$15.20, 16.8, 17.6$]{TW02}):
		\begin{align*}
			[u_1, u_3] = u_2, \,\, [u_3, u_5] = u_4, \,\, [u_1, u_5] = u_2 u_4, \,\, [u_2, u_6] = u_4, \,\, [u_1, u_6] = u_2 u_3 u_4 u_5
		\end{align*}
		We define $M_{\{\beta_1, \beta_3\}} := \{ \beta_2 \}, M_{\{\beta_3, \beta_5\}} := \{ \beta_4 \}, M_{\{\beta_1, \beta_5\}} := \{ \beta_2, \beta_4 \}, M_{\{\beta_2, \beta_6\}} := \{ \beta_4 \}$ and $M_{\{\beta_1, \beta_6\}} := \{ \beta_2, \beta_3, \beta_4, \beta_5 \}$.
	\end{enumerate}
	Note that for $i<j$ we have $[u_i, u_j] = \prod\nolimits_{\gamma \in M_{\{\beta_i, \beta_j\}}} u_{\gamma}$, where the order of the product is given by the order of the indices. For $i>j$ we have $[u_i, u_j] = \prod\nolimits_{\gamma \in M_{\{\beta_i, \beta_j\}}} u_{\gamma}$, where the order of the product is given by the inverse order. Thus $M_{\{\beta_i, \beta_j\}}$ contains all information about the commutators $[u_i, u_j]$ and $[u_j, u_i]$.
\end{example}

\section{Commutator blueprints}\label{Section: commutator blueprints}

In this section we will define commutator blueprints. These objects prescribe the commutation relations between prenilpotent pairs of roots. In view of Example \ref{exprank2rgd} we have a symmetry in the simple roots of the commutation relations except in the case $G_2$. To ensure that our definition is well-defined, we make the following convention:

\begin{convention}
	For the rest of this paper we assume $m_{st} \in \{2, 3, 4, 6, \infty\}$ for all $s\neq t \in S$. Moreover, we assume that every edge in the Coxeter diagram labeled with $6$ has a direction.
\end{convention}

We let $\mathcal{P}$ be the set of prenilpotent pairs of positive roots. For $w\in W$ we define $\Phi(w) := \{ \alpha \in \Phi_+ \mid w \notin \alpha \}$. Let $G = (c_0, \ldots, c_k) \in \mathrm{Min}$ and let $(\alpha_1, \ldots, \alpha_k)$ be the sequence of roots crossed by $G$. We define $\Phi(G) := \{ \alpha_i \mid 1 \leq i \leq k \}$. Using the indices we obtain an ordering $\leq_G$ on $\Phi(G)$ and, in particular, on $[\alpha, \beta] = [\beta, \alpha] \subseteq \Phi(G)$ for all $\alpha, \beta \in \Phi(G)$. Note that $\Phi(G) = \Phi(w)$ holds for every $G \in \mathrm{Min}(w)$. We abbreviate $\mathcal{I} := \{ (G, \alpha, \beta) \in \mathrm{Min} \times \Phi_+ \times \Phi_+ \mid \alpha, \beta \in \Phi(G), \alpha \leq_G \beta \}$. 

Given a family $\left(M_{\alpha, \beta}^G \right)_{(G, \alpha, \beta) \in \mathcal{I}}$, where $M_{\alpha, \beta}^G \subseteq (\alpha, \beta)$ is ordered via $\leq_G$. For $w\in W$ we define the group $U_w$ via the following presentation:
\[ U_w := \left\langle \{ u_{\alpha} \mid \alpha \in \Phi(w) \} \;\middle|\; \begin{cases*}
		\forall \alpha \in \Phi(w): u_{\alpha}^2 = 1, \\
		\forall (G, \alpha, \beta) \in \mathcal{I}, G \in \mathrm{Min}(w): [u_{\alpha}, u_{\beta}] = \prod\nolimits_{\gamma \in M_{\alpha, \beta}^G} u_{\gamma}
	\end{cases*} \right\rangle
\]
Here the product $\prod_{\gamma \in M_{\alpha, \beta}^G} u_{\gamma}$ is understood to be ordered via the ordering $\leq_G$, i.e.\ if $(G, \alpha, \beta) \in \mathcal{I}$ with $G \in \mathrm{Min}(w)$ and $M_{\alpha, \beta}^G = \{ \gamma_1 \leq_G \ldots \leq_G \gamma_k \} \subseteq (\alpha, \beta) \subseteq \Phi(G)$, then $\prod\nolimits_{\gamma \in M_{\alpha, \beta}^G} u_{\gamma} = u_{\gamma_1} \cdots u_{\gamma_k}$. Note that there could be $G, H \in \mathrm{Min}(w), \alpha, \beta \in \Phi(w)$ with $\alpha \leq_G \beta$ and $\beta \leq_H \alpha$. In this case we have two commutation relations, namely 
\begin{align*}
	&[u_{\alpha}, u_{\beta}] = \prod\nolimits_{\gamma \in M_{\alpha, \beta}^G} u_{\gamma} &&\text{and} &&[u_{\beta}, u_{\alpha}] = \prod\nolimits_{\gamma \in M_{\beta, \alpha}^H} u_{\gamma}
\end{align*}
From now on we will implicitly assume that each product $\prod\nolimits_{\gamma \in M_{\alpha, \beta}^G} u_{\gamma}$ is ordered via the ordering $\leq_G$.

\begin{definition}
	A \emph{commutator blueprint of type $(W, S)$} is a family $\mathcal{M} = \left(M_{\alpha, \beta}^G \right)_{(G, \alpha, \beta) \in \mathcal{I}}$ of subsets $M_{\alpha, \beta}^G \subseteq (\alpha, \beta)$ ordered via $\leq_G$ satisfying the following axioms:
	\begin{enumerate}[label=(CB\arabic*)]
		\item Let $G = (c_0, \ldots, c_k) \in \mathrm{Min}$ and let $H = (c_0, \ldots, c_m)$ for some $1 \leq m \leq k$. Then $M_{\alpha, \beta}^H = M_{\alpha, \beta}^G$ holds for all $\alpha, \beta \in \Phi(H)$ with $\alpha \leq_H \beta$.
		
		\item Suppose $s\neq t \in S$ with $m := m_{st} < \infty$. Let $G \in \mathrm{Min}(r_{\{s, t\}})$, let $(\alpha_1, \ldots, \alpha_m)$ be the sequence of roots crossed by $G$ and let $1 \leq i < j \leq m$. If $m_{st} \neq 6$, then we have
		\[ M_{\alpha_i, \alpha_j}^G = \begin{cases}
			(\alpha_i, \alpha_j) & \{ \alpha_i, \alpha_j \} = \{ \alpha_s, \alpha_t \} \\
			\emptyset & \{ \alpha_i, \alpha_j \} \neq \{ \alpha_s, \alpha_t \}
		\end{cases} \]
		If $m_{st} = 6$ and if $(t, s) \in E(S)$ and $G \in \mathrm{Min}_s(r_{\{s, t\}})$, then $M_{\alpha_i, \alpha_j}^G = M_{\{\alpha_i, \alpha_j\}}$ as sets, where $M_{\{\alpha_i, \alpha_j\}}$ is given in Example \ref{exprank2rgd}.
		
		\item For each $w\in W$ we have $\vert U_w \vert = 2^{\ell(w)}$, where $U_w$ is defined as above.
	\end{enumerate}
\end{definition}

\begin{remark}\label{Remark: Product mapping is a bijection}
	 In (CB$1$) we have $\Phi(H) \subseteq \Phi(G)$ and the order $\leq_G$ restricted to elements in $\Phi(H)$ is precisely the order $\leq_H$. Thus the expression $M_{\alpha, \beta}^G$ is defined. In (CB$2$) we have $\Phi(G) = [\alpha_s, \alpha_t]$ and we only require that $M_{\alpha, \beta}^G = M_{\{\alpha, \beta\}}$ as sets. Note that $M_{\alpha, \beta}^G$ is an ordered set and the axiom only makes a statement about the underlying set. Moreover, we note that the connection between the direction of the edge in the Coxeter diagram and the commutation relations corresponds to the connection of the commutation relations and the usual \emph{Dynkin diagrams}.
\end{remark}

\begin{example}\label{expintrgrable}
	Let $\mathcal{D} = (G, (U_{\alpha})_{\alpha \in \Phi})$ be an RGD-system of type $(W, S)$ over $\FF_2$, let $H = (c_0, \ldots, c_k) \in \mathrm{Min}$ and let $(\alpha_1, \ldots, \alpha_k)$ be the sequence of roots crossed by $H$. Then we have $\Phi(H) = \{ \alpha_1 \leq_H \cdots \leq_H \alpha_k \}$. By \cite[Corollary $8.34(1)$]{AB08} there exists for each $1 \leq m < i < n \leq k$ a unique $\epsilon_i \in \{ 0, 1 \}$ such that $[u_{\alpha_m}, u_{\alpha_n}] = \prod\nolimits_{i=m+1}^{n-1} u_{\alpha_i}^{\epsilon_i}$ holds, and $\epsilon_i = 1$ implies $\alpha_i \in (\alpha_m, \alpha_n)$. We define $ M(\mathcal{D})_{\alpha_m, \alpha_n}^H := \{ \alpha_i \in \Phi(H) \mid [u_{\alpha_m}, u_{\alpha_n}] = \prod\nolimits_{i=m+1}^{n-1} u_{\alpha_i}^{\epsilon_i}, \epsilon_i = 1 \} \subseteq (\alpha_m, \alpha_n)$ and $\mathcal{M}_{\mathcal{D}} := \left( M(\mathcal{D})_{\alpha, \beta}^H \right)_{(H, \alpha, \beta) \in \mathcal{I}}$. 
	
	For $s, t \in S$ with $m_{st} = 6$ we get a canonical direction of the edge $\{s, t\}$ via the commutation relations. Clearly, (CB$1$) is satisfied. By Example \ref{exprank2rgd}, (CB$2$) holds and (CB$3$) is satisfies by \cite[Corollary $8.34(1)$]{AB08}. Thus $\mathcal{M}_{\mathcal{D}}$ is a commutator blueprint of type $(W, S)$.
\end{example}

\begin{convention}
	From now on we let $\mathcal{M} = \left(M_{\alpha, \beta}^G \right)_{(G, \alpha, \beta) \in \mathcal{I}}$ be a commutator blueprint of type $(W, S)$.
\end{convention}

\begin{lemma}\label{Lemma: Definition of UG}
	Let $w\in W, G = (c_0, \ldots, c_k) \in \mathrm{Min}(w)$ and let $(\alpha_1, \ldots, \alpha_k)$ be the sequence of roots crossed by $G$. Then the group $U_w$ has the following presentation:
	\[ U_G := \left\langle u_{\alpha_1}, \ldots, u_{\alpha_k} \;\middle|\; \forall 1 \leq i \leq j \leq k: \quad u_{\alpha_i}^2 = 1, \quad [u_{\alpha_i}, u_{\alpha_j}] = \prod\nolimits_{\gamma \in M_{\alpha_i, \alpha_j}^G} u_{\gamma} \right\rangle \]
\end{lemma}
\begin{proof}
	Note that $\Phi(w) = \Phi(G) = \{ \alpha_1, \ldots, \alpha_k \}$. Clearly, we have an epimorphism $U_G \to U_w, \alpha_i \mapsto \alpha_i$. Since each element in $U_G$ is of the form $\prod\nolimits_{i=1}^{k} u_{\alpha_i}^{\epsilon_i}$, where $\epsilon_i \in \{0, 1\}$, $U_G$ has cardinality at most $2^k$. As $U_w$ has cardinality $2^k$, the claim follows.
\end{proof}

\begin{definition}
	Using the previous lemma, the axioms (CB$1$) and (CB$3$) imply that the canonical mapping $u_{\alpha} \mapsto u_{\alpha}$ induces a monomorphism from $U_w$ to $U_{ws}$ for all $w\in W, s\in S$ with $\ell(ws) = \ell(w) +1$. We denote by $U_+$ the direct limit of the groups $U_w$ with natural inclusions $U_w \to U_{ws}$ if $\ell(ws) = \ell(w) +1$.
\end{definition}

\begin{definition}\label{Definition: Properties of M}
	\begin{enumerate}[label=(\alph*)]
		\item $\mathcal{M}$ is called \emph{Weyl-invariant} if for all $w\in W$, $s\in S$, $G \in \mathrm{Min}_s(w)$ and $\alpha, \beta \in \Phi(G) \backslash \{ \alpha_s \}$ with $\alpha \leq_G \beta$ we have $M_{s\alpha, s\beta}^{sG} = sM_{\alpha, \beta}^G := \{ s\gamma \mid \gamma \in M_{\alpha, \beta}^G \}$.
		
		\item $\mathcal{M}$ is called \emph{faithful}, if the canonical homomorphisms $U_w \to U_+$ are injective.
		
		\item $\mathcal{M}$ is called \emph{integrable} if there exists an RGD-system $\mathcal{D}$ of type $(W, S)$ over $\FF_2$ such that the two families $\mathcal{M}$ and $\mathcal{M}_{\mathcal{D}}$ coincide pointwise.
	\end{enumerate}
\end{definition}

\begin{remark}\label{Remark: Weyl-invariance}
	\begin{enumerate}[label=(\alph*)]
		\item Let $\mathcal{D}$ be an RGD-system of type $(W, S)$ over $\FF_2$. Then $\mathcal{M}_{\mathcal{D}}$ is faithful (cf.\ \cite[Theorem $8.85$]{AB08}) and Weyl-invariant.
		
		\item Suppose that $\mathcal{M}$ is Weyl-invariant. Let $w \in W, s\in S, G \in \mathrm{Min}_s(w)$ and let $\alpha \neq \beta \in \Phi(G) \backslash \{ \alpha_s \}$. Then $\alpha \leq_G \beta$ if and only if $s\alpha \leq_{sG} s\beta$. Moreover, we have the following relation in $U_{sw}$:
		\allowdisplaybreaks
		\begin{align*}
			[u_{s\alpha}, u_{s\beta}] = \prod\nolimits_{\gamma \in M_{s\alpha, s\beta}^{sG}} u_{\gamma} = \prod\nolimits_{\gamma \in sM_{\alpha, \beta}^G} u_{\gamma} = \prod\nolimits_{\gamma \in M_{\alpha, \beta}^G} u_{s\gamma}
		\end{align*}
	\end{enumerate}
\end{remark}

\begin{lemma}\label{Lemma: V_{w, s} to U_{sw} isomorphism}
	For $w\in W, s\in S$ with $\ell(sw) = \ell(w) -1$ we let $V_{w, s}$ be the normal subgroup of $U_w$ generated by $\{ u_{\alpha} \mid \alpha \in \Phi(w) \backslash \{ \alpha_s \}  \}$. Let $G = (c_0, \ldots, c_k) \in \mathrm{Min}_s(w)$ and let $(\alpha_1 = \alpha_s, \ldots, \alpha_k)$ be the sequence of roots crossed by $G$. Then we define the group
	\[ V_G := \left\langle u_{\alpha_2}, \ldots, u_{\alpha_k} \;\middle|\; \forall 2 \leq i \leq j \leq k: \quad u_{\alpha_i}^2 = 1, \quad [u_{\alpha_i}, u_{\alpha_j}] = \prod\nolimits_{\gamma \in M_{\alpha_i, \alpha_j}^G} u_{\gamma} \right\rangle. \]
	\begin{enumerate}[label=(\alph*)]
		\item The canonical mapping $u_{\alpha_i} \mapsto u_{\alpha_i}$ extends to an isomorphism from $V_G$ to $V_{w, s}$.
		
		\item If $\mathcal{M}$ is Weyl-invariant, $u_{\alpha} \mapsto u_{s\alpha}$ extends to an isomorphism from $V_{w, s}$ to $U_{sw}$.
	\end{enumerate}
\end{lemma}
\begin{proof}
	Using the commutation relations and the fact that $[u_{\alpha_s}, u_{\alpha}] = u_{\alpha}^{u_{\alpha_s}} u_{\alpha}$, the subgroup $V_{w, s}$ is a normal subgroup of $U_w$. The first part follows similar as in Lemma \ref{Lemma: Definition of UG}. For the second part we note that $sG \in \mathrm{Min}(sw)$. Using Lemma \ref{Lemma: Definition of UG} and Remark \ref{Remark: Weyl-invariance}$(b)$, we obtain that the mapping $u_{\alpha} \to u_{s\alpha}$ extends to an isomorphism.
\end{proof}

\section{Construction of the rank $1$ parabolics}\label{Section: Constructing the rank 1 parabolics}

\begin{convention}
	From now on we assume that the commutator blueprint $\mathcal{M}$ is faithful and Weyl-invariant. Moreover, we fix $s\in S$ in this section, unless it is stated otherwise.
\end{convention}

As we have seen in Remark \ref{Remark: Weyl-invariance}$(a)$, an integrable commutator blueprint is necessarily faithful and Weyl-invariant. We will show (cf.\ Theorem \ref{RGDsystem}) that any faithful and Weyl-invariant commutator blueprint is already integrable by constructing an RGD-system which contains the group $U_+$ as a subgroup. As a first step we construct the group $P_s$ (cf.\ Definition \ref{Definition: Ps} and \ref{Definition: direct limit G}), which contains $U_+$ as a subgroup.

Since $\mathcal{M}$ is faithful, we can identify $U_w$ with its image in $U_+$. In particular, we have $u_{\alpha} \in U_+$ for all $\alpha \in \Phi_+$. We will write for short $u_s := u_{\alpha_s}$. We define the subgroup $N_s := \langle x^{-1} u_{\alpha} x \mid \alpha \in \Phi_+ \backslash \{ \alpha_s \}, x\in U_s \rangle \leq U_+$ (the idea of the definition of $N_s$ is obtained from \cite[$6.2.1$]{Re03}). Next, we will construct two automorphisms of $N_s$. Clearly, $U_+$ is generated by $U_s$ and $N_s$, and $N_s$ is a normal subgroup of $U_+$.

\begin{lemma}\label{UplusNsUs}
	We have $U_+ = U_s \ltimes N_s$.
\end{lemma}
\begin{proof}
	It suffices to show that $U_s \cap N_s = 1$. We first show that the assignments $u_{\alpha} \mapsto 1$ for $\alpha_s \neq \alpha \in \Phi_+$ and $u_s \mapsto u_s$ extend to a homomorphism $U_w \to U_s$. In view of the definition of $U_w$ it suffices to consider the relations $u_{\alpha}^2 = 1$ and $[u_{\alpha}, u_{\beta}] = u_{\gamma_1} \cdots u_{\gamma_k}$. Since $\alpha_s \notin (\alpha, \beta)$ for all $\{ \alpha, \beta \} \in \P$, these relations are mapped to $1$ and we obtain homomorphisms $U_w \to U_s$ for every $w\in W$. Since these homomorphisms respect the natural inclusions $U_w \to U_{wt}$, the universal property of direct limits yields a homomorphism $\phi: U_+ \to U_s$ with $\phi(u_{\alpha}) = 1$ for $\alpha_s \neq \alpha \in \Phi_+$ and $\phi(u_s) = u_s$. Since $N_s \leq \ker \phi$ and $U_s \cap \ker \phi = 1$, the claim follows.
\end{proof}

\begin{remark}
	The next step is to construct an automorphism $\tau_s$ on $N_s$ which maps $u_{\alpha}$ to $u_{s\alpha}$. The rough idea is that $P_s$ should look like $\langle u_s, \tau_s \rangle \ltimes N_s$ (cf.\ Definition \ref{Definition: Ps}). In the next lemma we will show that $N_s$ has a suitable presentation. The elements $v_{\alpha}$ will play the role of the elements $u_s u_{\alpha} u_s$ for all $\alpha_s \neq \alpha \in \Phi_+$.
\end{remark}

\begin{lemma}\label{Lemma: Presentation of Ns}
	We define the group $M_s$ via the following presentation:
	\[ \left\langle \{ u_{\alpha}, v_{\alpha} \mid \alpha_s \neq \alpha \in \Phi_+ \} \; \middle| \; \begin{cases*}
		\forall \alpha_s \neq \alpha \in \Phi_+: u_{\alpha}^2 = 1 = v_{\alpha}^2, \\
		\forall w\in W, G \in \mathrm{Min}_s(w), \alpha \leq_G \beta \in \Phi(G) \backslash \{ \alpha_s \}: \\
		\qquad [u_{\alpha}, u_{\beta}] = \prod\nolimits_{\gamma \in M_{\alpha, \beta}^G} u_{\gamma}, \quad [v_{\alpha}, v_{\beta}] = \prod\nolimits_{\gamma \in M_{\alpha, \beta}^G} v_{\gamma}, \\
		\forall w\in W, \ell(sw) = \ell(w) -1, G \in \mathrm{Min}_s(w), \alpha_s \neq \alpha \in \Phi(G): \\ \qquad v_{\alpha} = \left( \prod\nolimits_{\gamma \in M_{\alpha_s, \alpha}^G} u_{\gamma} \right) u_{\alpha}
	\end{cases*} \right\rangle \]
	Then we have $u_s \in \Aut(M_s)$ such that $u_s(u_{\alpha}) = v_{\alpha}$ and $u_s(v_{\alpha}) = u_{\alpha}$. In particular, 
	\[ M_s \to N_s, \begin{cases*}
		u_{\alpha} \mapsto u_{\alpha} \\
		v_{\alpha} \mapsto u_s u_{\alpha} u_s
	\end{cases*} \]
	is an isomorphism.
\end{lemma}
\begin{proof}
	We show that the assignments $u_{\alpha} \mapsto v_{\alpha}$ and $v_{\alpha} \mapsto u_{\alpha}$ extend to an endomorphism of $M_s$. Therefore we have to show that every relation is mapped to a relation. For that it suffices to consider the relations of the form $v_{\alpha} = \left( \prod\nolimits_{\gamma \in M_{\alpha_s, \alpha}^G} u_{\gamma} \right) u_{\alpha}$. Suppose $w\in W$ with that $\ell(sw) = \ell(w) -1$ and let $G \in \mathrm{Min}_s(w)$. Using Lemma \ref{Lemma: V_{w, s} to U_{sw} isomorphism} we deduce that the canonical assignment $u_{\alpha} \mapsto u_{\alpha}$ extends to a homomorphism from $V_{w, s} \cong V_G$ to $M_s$. Moreover, for $\alpha_s \neq \alpha \in \Phi(G)$ we have the following relation in $U_w$ (note that $\alpha \in \Phi(G)$ implies $\gamma \in \Phi(G)$ for all $\gamma \in (\alpha_s, \alpha)$):
	\allowdisplaybreaks
	\begin{align*}
		\left( \prod\nolimits_{\gamma \in M_{\alpha_s, \alpha}^G} \left( \prod\nolimits_{\beta \in M_{\alpha_s, \gamma}^G} u_{\beta} \right) u_{\gamma} \right) \left( \prod\nolimits_{\gamma \in M_{\alpha_s, \alpha}^G} u_{\gamma} \right) u_{\alpha} &= \left( \prod\nolimits_{\gamma \in M_{\alpha_s, \alpha}^G} [u_s, u_{\gamma}] u_{\gamma} \right) [u_s, u_{\alpha}] u_{\alpha} \\	
		&= u_s \left( \prod\nolimits_{\gamma \in M_{\alpha_s, \alpha}^G} u_{\gamma} \right) u_{\alpha} u_s \\
		&= u_s [u_s, u_{\alpha}] u_{\alpha} u_s \\
		&= u_{\alpha}
	\end{align*}
	Since both sides of the equation are contained in $V_{w, s}$, this is also a relation in $M_s$. Note that by definition we also have the relation $v_{\delta} = \left( \prod\nolimits_{\epsilon \in M_{\alpha_s, \delta}^G} u_{\epsilon} \right) u_{\delta}$ for each $\alpha_s \neq \delta \in \Phi(G)$ in $M_s$. Now we consider the discussed relation:
	\allowdisplaybreaks
	\begin{align*}
		\left( \prod\nolimits_{\gamma \in M_{\alpha_s, \alpha}^G} v_{\gamma} \right) v_{\alpha} &= \left( \prod\nolimits_{\gamma \in M_{\alpha_s, \alpha}^G} \left( \prod\nolimits_{\beta \in M_{\alpha_s, \gamma}^G} u_{\beta} \right) u_{\gamma} \right) \left( \prod\nolimits_{\gamma \in M_{\alpha_s, \alpha}^G} u_{\gamma} \right) u_{\alpha} = u_{\alpha}
	\end{align*}
	Thus every relation is mapped to a relation and we have an endomorphism $u_s$ of $M_s$ interchanging $u_{\alpha}$ and $v_{\alpha}$. Since $u_s^2 = \id$, it is an automorphism of $M_s$. Consider $U := \ZZ_2 \ltimes M_s$, where $\ZZ_2$ acts on $M_s$ via $u_s$. Moreover, we denote the generator of $\ZZ_2$ by $u_s$. Then the assignment
	\allowdisplaybreaks
	\begin{align*}
		&u_s \mapsto u_s, &&u_{\alpha} \mapsto u_{\alpha}, &&v_{\alpha} \mapsto u_s u_{\alpha} u_s
	\end{align*}
	extends to a homomorphism $U \to U_+$, since all relations in $U$ do also hold in $U_+$. Now we will show that there does also exist a homomorphism $U_+ \to U$ mapping $u_s$ onto $u_s$ and $u_{\alpha}$ onto $u_{\alpha}$. For this we consider $w\in W$. If $\ell(sw) = \ell(w) +1$, then every relation in $U_w$ is also a relation in $M_s$ and hence in $U$. Thus we obtain a homomorphism $U_w \to U$ mapping $u_{\alpha}$ onto $u_{\alpha}$. Assume that $\ell(sw) = \ell(w) -1$ and let $G \in \mathrm{Min}_s(w)$. By Lemma \ref{Lemma: Definition of UG} $U_w$ is isomorphic to $U_G$ and we have to show that $[u_s, u_{\alpha}] = \prod\nolimits_{\gamma \in M_{\alpha_s, \alpha}^G} u_{\gamma}$ is a relation in $U$. Note that this is a relation if and only if $u_s u_{\alpha} u_s = \left( \prod\nolimits_{\gamma \in M_{\alpha_s, \alpha}^G} u_{\gamma} \right) u_{\alpha}$ is a relation in $U$. But in $U$ we have $u_s u_{\alpha} u_s = v_{\alpha}$ and hence it is a relation by definition. In particular, the mappings $U_w \to U$ preserve the inclusions $U_w \to U_{wt}$ and by the universal property of direct limits there exists a homomorphism $U_+ \to U$. Since both concatenations are the identity on the generating sets, both homomorphisms are isomorphisms. In particular, $M_s$ is isomorphic to $N_s$.
\end{proof}

\begin{lemma}\label{tausV}
	Let $R \in \partial^2 \alpha_s$ and let $\Phi(R) := \{ \alpha \in \Phi_+ \mid R \in \partial^2 \alpha \}$. We define the group $U_R$ via the following presentation
	\[ U_R := \left\langle \{ u_{\alpha} \mid \alpha \in \Phi(R) \} \; \middle| \; \begin{cases*}
		\forall w\in W, G \in \mathrm{Min}_s(w), \alpha, \beta \in \Phi(G) \cap \Phi(R), \alpha \leq_G \beta: \\ \qquad u_{\alpha}^2 = 1, \qquad [u_{\alpha}, u_{\beta}] = \prod\nolimits_{\gamma \in M_{\alpha, \beta}^G} u_{\gamma}
	\end{cases*} \right\rangle \]
	For $N_R := \langle u_{\alpha} \mid \alpha_s \neq \alpha \in \Phi(R) \rangle \leq U_R$ we have $U_R \cong U_s \ltimes N_R$. Furthermore, there exists $\tau_s \in \Aut(N_R)$ such that $\tau_s(u_{\alpha}) = u_{s\alpha}$, and we have $\tau_s^2 = 1 = (u_s \tau_s)^3 $ in $\Aut(N_R)$.
\end{lemma}
\begin{proof}
	Note first that $[\alpha, \beta] \subseteq \Phi(R)$ for all $\alpha, \beta \in \Phi(R)$. Hence the relations of $U_R$ make sense. Similar as in Lemma \ref{UplusNsUs} we deduce $U_R \cong U_s \ltimes N_R$. Suppose $w\in W$ with $\ell(sw) = \ell(w) -1$ and let $G \in \mathrm{Min}_s(w)$ be such that $\Phi(R) \subseteq \Phi(G)$. Then each element of $U_R$ can be written in the form $\prod\nolimits_{j=1}^m u_{\beta_j}^{\epsilon_j}$, where $\epsilon_j \in \{0, 1\}$ and $\{ \beta_1 = \alpha_s \leq_G \cdots \leq_G \beta_m \} = \Phi(R) \subseteq \Phi(G)$. Since we have a homomorphism $U_R \to U_+$ and the image of $U_R$ is contained in $U_w$, (CB$3$) implies that $U_R \to U_+$ is a monomorphism. 
	
	Let $\tilde{N}_R$ be the group given by the presentation of $U_R$ by deleting the generator $u_{\alpha_s}$ and all relations in which $u_{\alpha_s}$ appears. Then again each element in $\tilde{N}_R$ can be written in the form $\prod\nolimits_{j=2}^m u_{\beta_j}^{\epsilon_j}$. Since we have a homomorphism $\tilde{N}_R \to U_R$ with image $N_R$, the cardinality of $N_R$ implies that this homomorphism must be an isomorphism. In particular, $\tilde{N}_R$ yields a presentation of $N_R$.
	
	Now we will show that the assignment $u_{\alpha} \mapsto u_{s\alpha}$ extends to an endomorphism of $N_R$. First of all we note that for $\alpha_s \neq \alpha \in \Phi(R)$ we have $\alpha_s \neq s\alpha \in \Phi(R)$. We have to check that every relation is mapped to a relation. We consider the two different types of relations and note that $u_{\alpha}^2 = 1$ is obvious. Suppose $w\in W, G\in \mathrm{Min}_s(w)$ and $\alpha, \beta \in \left( \Phi(G) \cap \Phi(R) \right) \backslash \{ \alpha_s \}$ with $\alpha \leq_G \beta$. Using the Weyl-invariance and the fact that $[u_{s\alpha}, u_{s\beta}] = \prod\nolimits_{\gamma \in M_{s\alpha, s\beta}^{sG}} u_{\gamma}$ is a relation in $N_R$, we deduce as in Remark \ref{Remark: Weyl-invariance}$(b)$ that
	\allowdisplaybreaks
	\begin{align*}
		[u_{s\alpha}, u_{s\beta}] &= \prod\nolimits_{\gamma \in M_{\alpha, \beta}^G} u_{s\gamma}
	\end{align*}
	is also a relation in $N_R$. Thus $\tau_s: N_R \to N_R, u_{\alpha} \mapsto u_{s\alpha}$ is an endomorphism. Since $\tau_s^2 = 1$, we infer $\tau_s \in \Aut(N_R)$.	It is left to show that $\left( u_s \tau_s \right)^3 = 1$ holds in $\Aut(N_R)$. Therefore, we do a case distinction on the type of the residue $R$ (we will write for short $f.u_{\beta} := f(u_{\beta})$):
	\begin{itemize}[label=$\bullet$]
		\item $A_1 \times A_1$: Let $\Phi(R) = \{ \alpha_s, \beta \}$. Then $s\beta = \beta$. Since $u_s, u_{\beta}$ commute by (CB$2$), Example \ref{exprank2rgd} and the Weyl-invariance, we obtain 
		\[ (u_s \tau_s)^3.u_{\beta} = (u_s \tau_s)^2. [u_s, u_{\beta}] u_{\beta} = (u_s \tau_s)^2.u_{\beta} = u_{\beta} \]
		
		\item $A_2$: Let $\Phi(R) = \{ \alpha_s, \delta, \epsilon \}$. Then $s\epsilon = \delta$ and we assume that $\{ \alpha_s, \epsilon \}$ is a set of \emph{simple roots of $R$} (i.e.\ $\Phi(R) \subseteq [\alpha_s, \epsilon]$). Using (CB$2$), Example \ref{exprank2rgd} and the Weyl-invariance, we obtain the following:
		\begin{align*}
			(u_s \tau_s)^3.u_{\epsilon} &= (u_s \tau_s)^2. u_{\delta} = (u_s \tau_s). u_{\delta}u_{\epsilon} = u_{\epsilon} \\
			(u_s \tau_s)^3.u_{\delta} &= u_s \tau_s. u_{\epsilon} = u_{\delta}
		\end{align*}
		
		\item $B_2 = C_2$: Let $\Phi(R) = \{ \alpha_s, \delta, \gamma, \epsilon \}$ and assume that $\{ \alpha_s, \epsilon \}$ is a set of simple roots of $R$. Furthermore, we assume that $s\gamma = \gamma$ and $s\epsilon = \delta$. Using (CB$2$), Example \ref{exprank2rgd} and the Weyl-invariance, we obtain that only $u_s$ and $u_{\epsilon}$ do not commute. We compute the following:
		\begin{align*}
			(u_s \tau_s)^3. u_{\gamma} &= (u_s \tau_s)^2. u_{\gamma} = u_{\gamma} \\
			(u_s \tau_s)^3. u_{\epsilon} &= (u_s \tau_s)^2. u_{\delta} = u_s \tau_s. u_{\delta}u_{\gamma}u_{\epsilon} = u_{\epsilon} \\
			(u_s \tau_s)^3. u_{\delta} &= u_s \tau_s. u_{\epsilon} = u_{\delta}
		\end{align*}
		
		\item $G_2$: Let $\Phi(R) = \{ \beta_1, \ldots, \beta_6 \}$ and we assume that $\{ \beta_1, \beta_6 \}$ is a set of simple roots of $R$ and that the roots are ordered via their indices. Assume first that $\alpha_s = \beta_1$. Then $s\beta_2 = \beta_6, s\beta_3 = \beta_5$ and $s\beta_4 = \beta_4$. Let $u_i := u_{\beta_i}$. Using (CB$2$), Example \ref{exprank2rgd} and the Weyl-invariance, we obtain
		\allowdisplaybreaks
		\begin{align*}
			(u_s \tau_s)^3. u_4 &= (u_s \tau_s)^2. u_4 = u_4 \\
			(u_s \tau_s)^3. u_6 &= (u_s \tau_s)^2. u_2 = u_s \tau_s. [u_1, u_6] u_6 = u_s\tau_s. u_2 u_3 u_4 u_5 u_6 \\
			&= [u_1, u_6] u_6 [u_1, u_5] u_5 [u_1, u_4] u_4 [u_1, u_3] u_3 [u_1, u_2] u_2 \\
			&= u_2 u_3 u_4 u_5 u_6 u_2 u_4 u_5 u_4 u_2 u_3 u_2 = u_2 u_3 u_4 u_6 u_3 u_2 = u_6 \\
			(u_s \tau_s)^3. u_2 &= u_s \tau_s. u_6 = u_2 \\
			(u_s \tau_s)^3. u_5 &= (u_s \tau_s)^2. [u_1, u_3] u_3 = (u_s \tau_s)^2. u_2 u_3 \\
			&= u_s \tau_s. [u_1, u_6] u_6 [u_1, u_5] u_5 \\
			&= u_s \tau_s. u_2 u_3 u_4 u_5 u_6 u_2 u_4 u_5 = u_s \tau_s. u_3 u_4 u_6 \\
			&= [u_1, u_5]u_5 [u_1, u_4] u_4 [u_1, u_2] u_2 = u_2 u_4 u_5 u_4 u_2 = u_5 \\
			(u_s \tau_s)^3. u_3 &= (u_s \tau_s)^2. [u_1, u_5] u_5 = (u_s \tau_s)^2. u_2 u_4 u_5 = u_6 u_4 u_3 u_4 u_6 = u_3
		\end{align*}
		It is also possible that $\alpha_s = \beta_6$. In this case $s\beta_1 = \beta_5, s\beta_2 = \beta_4$ and $s\beta_3 = \beta_3$ and we compute the following:
		\allowdisplaybreaks
		\begin{align*}
			(u_s \tau_s)^3. u_3 &= (u_s \tau_s)^2. u_3 = u_3 \\
			(u_s \tau_s)^3. u_1 &= (u_s \tau_s)^2. u_5 = u_s \tau_s. u_1 [u_1, u_6] = u_s \tau_s. u_1 u_2 u_3 u_4 u_5 \\
			&= u_5 [u_5, u_6] u_4 [u_4, u_6] u_3 [u_3, u_6] u_2 [u_2, u_6] u_1 [u_1, u_6] \\
			&= u_5 u_4 u_3 u_2 u_4 u_1 u_2 u_3 u_4 u_5 = u_5 u_4 u_1 u_2 u_5 = u_4 u_1 [u_1, u_5] u_2 = u_1 \\			
			(u_s \tau_s)^3. u_5 &= u_s \tau_s. u_1 = u_5 \\
			(u_s \tau_s)^3. u_2 &= (u_s \tau_s)^2. u_4 [u_4, u_6] = (u_s \tau_s)^2. u_4 \\
			&= u_s \tau_s. u_2 [u_2, u_6] = u_s \tau_s. u_2 u_4 \\
			&= u_4 [u_4, u_6] u_2 [u_2, u_6] = u_4 u_2 u_4 = u_2 \\
			(u_s \tau_s)^3. u_4 &= u_s \tau_s. u_2 = u_4 [u_4, u_6] = u_4 \qedhere
		\end{align*}
	\end{itemize}
\end{proof}

\begin{lemma}\label{Lemma: ustaus V}
	Let $R \in \partial^2 \alpha_s$ and let $\alpha_s \neq \alpha \in \Phi(R)$. Let $G \in \mathrm{Min}_s(w)$ be a minimal gallery with $\Phi(R) \subseteq \Phi(G)$ for some $w\in W$. Then the following hold in $N_R$:
	\[ \left( \prod\nolimits_{\gamma \in M_{\alpha_s, s\alpha}^G} u_{s\gamma} \right) u_{\alpha} = \left( \prod\nolimits_{\gamma \in M_{\alpha_s, \alpha}^G} \left( \prod\nolimits_{\gamma' \in M_{\alpha_s, s\gamma}^G} u_{\gamma'} \right) u_{s\gamma} \right) \left( \prod\nolimits_{\gamma \in M_{\alpha_s, s\alpha}^G} u_{\gamma} \right) u_{s\alpha} \]
\end{lemma}
\begin{proof}
	This follows from the previous lemma and the fact that the left hand side is equal to $\tau_s u_s \tau_s(u_{\alpha})$ and the right hand side is equal to $u_s \tau_s u_s (u_{\alpha})$.
\end{proof}

\begin{lemma}\label{Lemma: taus independent of gallery}
	Suppose $w, w'\in W$ with $\ell(sw) = \ell(w) -1$ and $\ell(sw') = \ell(w')-1$. Let $G \in \mathrm{Min}_s(w), H \in \mathrm{Min}_s(w')$ and let $\alpha_s \neq \alpha \in \Phi(G) \cap \Phi(H)$. Then the following hold in $M_s$:
	\begin{enumerate}[label=(\alph*)]
		\item $\prod\nolimits_{\gamma \in M_{\alpha_s, \alpha}^G} u_{s\gamma} = \prod\nolimits_{\gamma \in M_{\alpha_s, \alpha}^H} u_{s\gamma}$;
		
		\item $\prod\nolimits_{\gamma \in M_{\alpha_s, \alpha}^G} v_{s\gamma} = \prod\nolimits_{\gamma \in M_{\alpha_s, \alpha}^H} v_{s\gamma}$.
	\end{enumerate}
\end{lemma}
\begin{proof}
	Assertion $(b)$ is a direct consequence of Assertion $(a)$ and the fact that $u_s$ is an automorphism of $M_s$ interchanging $u_{\alpha}$ and $v_{\alpha}$. Thus it suffices to show Assertion $(a)$. By definition we have the following two equations in $M_s$:
	\[ \left( \prod\nolimits_{\gamma \in M_{\alpha_s, \alpha}^G} u_{\gamma} \right) u_{\alpha} = v_{\alpha} = \left( \prod\nolimits_{\gamma \in M_{\alpha_s, \alpha}^H} u_{\gamma} \right) u_{\alpha} \]
	Using Lemma \ref{Lemma: Presentation of Ns} we infer that $\prod\nolimits_{\gamma \in M_{\alpha_s, \alpha}^G} u_{\gamma} = \prod\nolimits_{\gamma \in M_{\alpha_s, \alpha}^H} u_{\gamma}$ is a relation in $N_s \leq U_+$. We remark that $[\alpha_s, \alpha] \subseteq \Phi(G) \cap \Phi(H)$. Using the fact that $U_w \to U_+$ is injective and both sides of the relation are contained in $U_w$, we deduce that it is also a relation in $U_w$. Moreover, both sides are contained in the subgroup $V_{w, s} \leq U_w$ and Lemma \ref{Lemma: V_{w, s} to U_{sw} isomorphism} yields that
	\[ \prod\nolimits_{\gamma \in M_{\alpha_s, \alpha}^G} u_{s\gamma} = \prod\nolimits_{\gamma \in M_{\alpha_s, \alpha}^H} u_{s\gamma} \]
	is a relation in $U_{sw}$. As $U_{sw} \to M_s$ is a homomorphism, the claim follows.
\end{proof}

\begin{proposition}\label{Proposition: Existence of taus}
	There exists an endomorphism $\tau_s :N_s \to N_s$ such that $\tau_s(u_{\alpha}) = u_{s\alpha}$ holds for each $\alpha_s \neq \alpha \in \Phi_+$ and $\tau_s(u_s u_{\beta} u_s) = u_s \left( \prod\nolimits_{\gamma \in M_{\alpha_s, s\beta}^G} u_{s\gamma} \right) u_{\beta} u_s$ holds for each $-\alpha_s \subseteq \beta \in \Phi_+$, where $w\in W$ is such that $\ell(sw) = \ell(w) -1$ and $G \in \mathrm{Min}_s(w)$ with $s\beta \in \Phi(G)$.
\end{proposition}
\begin{proof}
	We will construct an endomorphism $\tau_s: M_s \to M_s$ and show that the induced endomorphism on $N_s$ is as required. First of all we will show that the following assignments (call it $\tau_s$) extend to an endomorphism of $M_s$, where $w\in W$ is such that $\ell(sw) = \ell(w) -1$:
	\allowdisplaybreaks
	\begin{align*}
		\forall \alpha_s \neq \alpha \in \Phi_+: u_{\alpha} &\mapsto u_{s\alpha} \\
		\forall \{\alpha_s, \alpha\} \in \P: v_{\alpha} &\mapsto \left( \prod\nolimits_{\gamma \in M_{\alpha_s, \alpha}^G} u_{s\gamma} \right) u_{s\alpha} &&\text{where } (G, \alpha_s, \alpha) \in \mathcal{I}, \, G \in \mathrm{Min}_s(w) \\
		-\alpha_s \subseteq \alpha: v_{\alpha} &\mapsto \left( \prod\nolimits_{\gamma \in M_{\alpha_s, s\alpha}^G} v_{s\gamma} \right) v_{\alpha} &&\text{where } (G, \alpha_s, s\alpha) \in \mathcal{I}, \, G \in \mathrm{Min}_s(w)
	\end{align*}
	 We remark that by Lemma \ref{Lemma: taus independent of gallery} the assignments do neither depend on $w\in W$ with $\ell(sw) = \ell(w) -1$ nor on the gallery $G \in \mathrm{Min}_s(w)$. We distinguish all relations:
	\begin{enumerate}[label=(\roman*)]
		\item $u_{\alpha}^2 = 1$: There is nothing to show.
		
		\item $v_{\alpha}^2 = 1$: We distinguish the following cases:
		\allowdisplaybreaks
		\begin{enumerate}[label=(\alph*)]
			\item $\{ \alpha_s, \alpha \} \in \P$: Suppose $w\in W$ with $\ell(sw) = \ell(w) -1$ and $G \in \mathrm{Min}_s(w)$ with $\alpha_s, \alpha \in \Phi(G)$. Then we have $\left( \left( \prod\nolimits_{\gamma \in M_{\alpha_s, \alpha}^G} u_{\gamma} \right) u_{\alpha} \right)^2 = \left( [u_s, u_{\alpha}] u_{\alpha} \right)^2 =1$ in $U_w$ and hence in $V_{w, s}$. This implies that
			\[ \left( \left( \prod\nolimits_{\gamma \in M_{\alpha_s, \alpha}^G} u_{s\gamma} \right) u_{s\alpha} \right)^2 \]
			is a relation in $U_{sw}$ by Lemma \ref{Lemma: V_{w, s} to U_{sw} isomorphism} and, using the homomorphism $U_{sw} \to M_s$, hence also in $M_s$. But this is exactly the image of $v_{\alpha}^2$ under the assignment $\tau_s$.
			
			\item $-\alpha_s \subseteq \alpha$: Suppose $w\in W$ with $\ell(sw) = \ell(w) -1$ and $G \in \mathrm{Min}_s(w)$ with $\alpha_s, s\alpha \in \Phi(G)$. We have to show that 
			\[ \left( \left( \prod\nolimits_{\gamma \in M_{\alpha_s, s\alpha}^G} v_{s\gamma} \right) v_{\alpha} \right)^2 \]
			is a relation in $M_s$. Clearly, $\alpha_s \neq s\alpha \in \Phi_+$ and $v_{s\alpha}^2$ is a relation by definition. Using Case $(a)$, we already know that
			\[ \left( \left( \prod\nolimits_{\gamma \in M_{\alpha_s, s\alpha}^G} u_{s\gamma} \right) u_{\alpha} \right)^2 \]
			is a relation in $M_s$. Since $u_s$ is an automorphism of $M_s$ interchanging $u_{\alpha}$ and $v_{\alpha}$ by Lemma \ref{Lemma: Presentation of Ns}, we obtain the relation
			\[ 1 = u_s \left( \left( \left( \prod\nolimits_{\gamma \in M_{\alpha_s, s\alpha}^G} u_{s\gamma} \right) u_{\alpha} \right)^2 \right) = \left( \left( \prod\nolimits_{\gamma \in M_{\alpha_s, s\alpha}^G} v_{s\gamma} \right) v_{\alpha} \right)^2  \]
		\end{enumerate}
		
		\item $[u_{\alpha}, u_{\beta}] = \prod\nolimits_{\gamma \in M_{\alpha, \beta}^G} u_{\gamma}$: Suppose $w\in W, G \in \mathrm{Min}_s(w)$ and $\alpha \leq_G \beta \in \Phi(G) \backslash \{\alpha_s\}$. Using the Weyl-invariance as in Remark \ref{Remark: Weyl-invariance}$(b)$, we deduce the following relation in $M_s$:
		\[ [u_{s\alpha}, u_{s\beta}] = \prod\nolimits_{\gamma \in M_{s\alpha, s\beta}^{sG}} u_{\gamma} = \prod\nolimits_{\gamma \in sM_{\alpha, \beta}^G} u_{\gamma} = \prod\nolimits_{\gamma \in M_{\alpha, \beta}^G} u_{s\gamma}  \]
		
		\item $[v_{\alpha}, v_{\beta}] = \prod\nolimits_{\gamma \in M_{\alpha, \beta}^G} v_{\gamma}$: Suppose $w\in W, G\in \mathrm{Min}_s(w)$ and $\alpha \leq_G \beta \in \Phi(G) \backslash \{\alpha_s\}$. We distinguish the following cases:		
		\begin{enumerate}[label=(\alph*\alph*)]
			\item $\ell(sw) = \ell(w)-1$: Note that $\{ \alpha_s, \delta \} \in \P$ for each $\alpha_s \neq \delta \in \Phi(G)$. We have to show that
			\[ \left[ \left( \prod\nolimits_{\gamma \in M_{\alpha_s, \alpha}^G} u_{s\gamma} \right) u_{s\alpha}, \left( \prod\nolimits_{\gamma \in M_{\alpha_s, \beta}^G} u_{s\gamma} \right) u_{s\beta} \right] = \prod\nolimits_{\gamma \in M_{\alpha, \beta}^G} \left( \prod\nolimits_{\delta \in M_{\alpha_s, \gamma}^G} u_{s\delta} \right) u_{s\gamma}  \]
			is a relation in $M_s$. Note that $[u_{\alpha}, u_{\beta}] = \prod\nolimits_{\gamma \in M_{\alpha, \beta}^G} u_{\gamma}$ is a relation in $U_w$ and $V_{w, s}$, and hence also the $u_s$-conjugate, which is given by
			\allowdisplaybreaks
			\begin{align*}
				\left[ \left( \prod\nolimits_{\gamma \in M_{\alpha_s, \alpha}^G} u_{\gamma} \right) u_{\alpha}, \left( \prod\nolimits_{\gamma \in M_{\alpha_s, \beta}^G} u_{\gamma} \right) u_{\beta} \right] &= [u_s u_{\alpha} u_s, u_s u_{\beta} u_s] \\
				&= u_s [u_{\alpha}, u_{\beta}] u_s \\
				&= u_s \left( \prod\nolimits_{\gamma \in M_{\alpha, \beta}^G} u_{\gamma} \right) u_s \\
				&= \prod\nolimits_{\gamma \in M_{\alpha, \beta}^G} \left( \prod\nolimits_{\delta \in M_{\alpha_s, \gamma}^G} u_{\delta} \right) u_{\gamma}
			\end{align*}
			Using Lemma \ref{Lemma: V_{w, s} to U_{sw} isomorphism} and the homomorphism $U_{sw} \to M_s$, the claim follows.
			
			\item $\ell(sw) = \ell(w) +1$: Then $\alpha_s \notin \Phi(G)$. Let $\delta \in \Phi(G)$. Then either $-\alpha_s \subseteq \delta$ or $o(r_{\alpha_s} r_{\delta}) < \infty$. We first observe the following: Suppose $o(r_{\alpha_s} r_{\delta}) < \infty$ with $R \in \partial^2 \alpha_s \cap \partial^2 \delta$, and $H \in \mathrm{Min}_s(w')$ with $\Phi(R) \subseteq \Phi(H)$ for some $w' \in W$. By Lemma \ref{Lemma: ustaus V} (applied to $\alpha = s\delta$) we have the following in $N_R$:
			\allowdisplaybreaks
			\begin{align*}
				\left( \prod\nolimits_{\gamma \in M_{\alpha_s, \delta}^H} u_{s\gamma} \right) u_{s\delta} &= \left( \prod\nolimits_{\gamma \in M_{\alpha_s, s\delta}^H} \left( \prod\nolimits_{\gamma' \in M_{\alpha_s, s\gamma}^H} u_{\gamma'} \right) u_{s\gamma} \right) \left( \prod\nolimits_{\omega \in M_{\alpha_s, \delta}^H} u_{\omega} \right) u_{\delta}
			\end{align*}
			Since we have a canonical homomorphism $N_R \to M_s$, this is also a relation in $M_s$. Combining this with Lemma \ref{Lemma: taus independent of gallery}$(b)$ and the fact that $v_{\rho} = \left( \prod\nolimits_{\omega \in M_{\alpha_s, \rho}^H} u_{\omega} \right) u_{\rho}$ is a relation in $M_s$ by definition for all $\rho \in \Phi(H) \backslash \{\alpha_s\}$, we deduce the following relation in $M_s$:
			\allowdisplaybreaks
			\begin{align*}
				\left( \prod\nolimits_{\gamma \in M_{\alpha_s, \delta}^H} u_{s\gamma} \right) u_{s\delta} &= \left( \prod\nolimits_{\gamma \in M_{\alpha_s, s\delta}^H} \left( \prod\nolimits_{\omega \in M_{\alpha_s, s\gamma}^H} u_{\omega} \right) u_{s\gamma} \right) \left( \prod\nolimits_{\omega \in M_{\alpha_s, \delta}^H} u_{\omega} \right) u_{\delta} \\
				&= \left( \prod\nolimits_{\gamma \in M_{\alpha_s, s\delta}^H} v_{s\gamma} \right) v_{\delta} \\
				&= \left( \prod\nolimits_{\gamma \in M_{\alpha_s, s\delta}^{sG}} v_{s\gamma} \right) v_{\delta}
			\end{align*}
			This shows that $v_{\delta}$ is mapped onto $\left( \prod\nolimits_{\gamma \in M_{\alpha_s, s\delta}^{sG}} v_{s\gamma} \right) v_{\delta}$ for each $\delta \in \Phi(G)$. In particular, this assignment does not depend on $o(r_{\alpha_s} r_{\delta})$ for $\delta \in \Phi(G)$. We have to verify that
			\[ \left[ \left( \prod\nolimits_{\gamma \in M_{\alpha_s, s\alpha}^{sG}} v_{s\gamma} \right) v_{\alpha}, \left( \prod\nolimits_{\gamma \in M_{\alpha_s, s\beta}^{sG}} v_{s\gamma} \right) v_{\beta} \right] = \prod\nolimits_{\gamma \in M_{\alpha, \beta}^G} \left( \prod\nolimits_{\delta \in M_{\alpha_s, s\gamma}^{sG}} v_{s\delta} \right) v_{\gamma}  \]
			is a relation in $M_s$. Note that $[v_{s\alpha}, v_{s\beta}] = \prod\nolimits_{\gamma \in M_{s\alpha, s\beta}^{sG}} v_{\gamma}$ is a relation in $M_s$. Using $(aa)$ we deduce that 
			\[ \left[ \left( \prod\nolimits_{\gamma \in M_{\alpha_s, s\alpha}^{sG}} u_{s\gamma} \right) u_{\alpha}, \left( \prod\nolimits_{\gamma \in M_{\alpha_s, s\beta}^{sG}} u_{s\gamma} \right) u_{\beta} \right] = \prod\nolimits_{\gamma \in M_{s\alpha, s\beta}^{sG}} \left( \prod\nolimits_{\delta \in M_{\alpha_s, \gamma}^{sG}} u_{s\delta} \right) u_{s\gamma} \]
			is a relation in $M_s$. Applying the automorphism $u_s \in \Aut(M_s)$ and using the Weyl-invariance we see that
			\allowdisplaybreaks
			\begin{align*}
				\left[ \left( \prod\nolimits_{\gamma \in M_{\alpha_s, s\alpha}^{sG}} v_{s\gamma} \right) v_{\alpha}, \left( \prod\nolimits_{\gamma \in M_{\alpha_s, s\beta}^{sG}} v_{s\gamma} \right) v_{\beta} \right] &= \prod\nolimits_{\gamma \in M_{s\alpha, s\beta}^{sG}} \left( \prod\nolimits_{\delta \in M_{\alpha_s, \gamma}^{sG}} v_{s\delta} \right) v_{s\gamma} \\
				&= \prod\nolimits_{\gamma \in sM_{\alpha, \beta}^{G}} \left( \prod\nolimits_{\delta \in M_{\alpha_s, \gamma}^{sG}} v_{s\delta} \right) v_{s\gamma} \\
				&= \prod\nolimits_{\gamma \in M_{\alpha, \beta}^{G}} \left( \prod\nolimits_{\delta \in M_{\alpha_s, s\gamma}^{sG}} v_{s\delta} \right) v_{\gamma}
			\end{align*}
			is a relation in $M_s$. This finishes the proof.
		\end{enumerate}
		
		\item $v_{\alpha} = \left( \prod\nolimits_{\gamma \in M_{\alpha_s, \alpha}^G} u_{\gamma} \right) u_{\alpha}$: This holds by definition.
	\end{enumerate}
	
	This shows the existence of the endomorphism $\tau_s:M_s \to M_s$. Using the isomorphism $\phi: M_s \to N_s$ from Lemma \ref{Lemma: Presentation of Ns}, we obtain an endomorphism $\tau_s: N_s \to N_s$ via $N_s \overset{\phi^{-1}}{\to} M_s \overset{\tau_s}{\to} M_s \overset{\phi}{\to} N_s$. Moreover, this endomorphism is as required.
\end{proof}

\begin{corollary}\label{ustaus}
	We have $\tau_s^2 = 1 = (u_s \tau_s)^3$. In particular, $\tau_s \in \Aut(N_s)$.
\end{corollary}
\begin{proof}
	For short we will not specify a gallery $G$. If $M_{\alpha_s, \alpha}^G$ appears, we will implicitly assume that $G \in \mathrm{Min}_s(w)$ for some $w\in W$ with $\ell(sw) = \ell(w) -1$ such that $\alpha \in \Phi(G)$.
	
	By the previous proposition we have $\tau_s(u_{\alpha}) = u_{s\alpha}$ for each $\alpha_s \neq \alpha \in \Phi_+$ and $\tau_s(u_s u_{\beta} u_s) = u_s \left( \prod\nolimits_{\gamma \in M_{\alpha_s, s\beta}^G} u_{s\gamma} \right) u_{\beta} u_s$ for each $-\alpha_s \subseteq \beta \in \Phi_+$. Using this we establish the claim. We will first show $\tau_s^2 = 1$. Therefore, let $\alpha_s \neq \alpha \in \Phi_+$. Then $\alpha_s \neq s\alpha \in \Phi_+$ and we have $\tau_s^2(u_{\alpha}) = \tau_s(u_{s\alpha}) = u_{\alpha}$. Now let $-\alpha_s \subseteq \beta \in \Phi_+$. Note that for $\gamma \in M_{\alpha_s, s\beta}^G$ we have $-\alpha_s \subseteq s\gamma$. This implies
	\allowdisplaybreaks
	\begin{align*}
		\tau_s^2(u_s u_{\beta} u_s) &= \tau_s \left( u_s \left( \prod\nolimits_{\gamma \in M_{\alpha_s, s\beta}^G} u_{s\gamma} \right) u_{\beta} u_s \right) \\
		&= \left( \prod\nolimits_{\gamma \in M_{\alpha_s, s\beta}^G} \tau_s( u_s u_{s\gamma} u_s ) \right) \tau_s( u_s u_{\beta} u_s ) \\
		&= \left(\prod\nolimits_{\gamma \in M_{\alpha_s, s\beta}^G} u_s \left( \prod\nolimits_{\delta \in M_{\alpha_s, \gamma}^G} u_{s\delta} \right) u_{s\gamma} u_s \right) \left( u_s \left( \prod\nolimits_{\gamma \in M_{\alpha_s, s\beta}^G} u_{s\gamma} \right) u_{\beta} u_s \right) \\
		&= u_s \left(\prod\nolimits_{\gamma \in M_{\alpha_s, s\beta}^G} \left( \prod\nolimits_{\delta \in M_{\alpha_s, \gamma}^G} u_{s\delta} \right) u_{s\gamma} \right) \left( \prod\nolimits_{\gamma \in M_{\alpha_s, s\beta}^G} u_{s\gamma} \right) u_{\beta} u_s
	\end{align*}
	Note that we have the following relation in $U_w$ and hence in $V_{w, s}$:
	\allowdisplaybreaks
	\begin{align*}
		\left(\prod\nolimits_{\gamma \in M_{\alpha_s, s\beta}^G} \left( \prod\nolimits_{\delta \in M_{\alpha_s, \gamma}^G} u_{\delta} \right) u_{\gamma} \right) \prod\nolimits_{\gamma \in M_{\alpha_s, s\beta}^G} u_{\gamma} &= \left(\prod\nolimits_{\gamma \in M_{\alpha_s, s\beta}^G} [u_s, u_{\gamma}] u_{\gamma} \right) [u_s, u_{s\beta}] \\
		&= u_s [u_s, u_{s\beta}] u_s [u_s, u_{s\beta}] \\
		&= (u_{s\beta}u_su_{s\beta})^2 = 1
	\end{align*}
	Using Lemma \ref{Lemma: V_{w, s} to U_{sw} isomorphism}, the following is a relation in $U_{sw}$ and hence in $N_s$:
	\allowdisplaybreaks
	\begin{align*}
		\left(\prod\nolimits_{\gamma \in M_{\alpha_s, s\beta}^G} \left( \prod\nolimits_{\delta \in M_{\alpha_s, \gamma}^G} u_{s\delta} \right) u_{s\gamma} \right) \left( \prod\nolimits_{\gamma \in M_{\alpha_s, s\beta}^G} u_{s\gamma} \right) = 1
	\end{align*}
	This shows $\tau_s^2(u_s u_{\beta} u_s) = u_s u_{\beta} u_s$ and hence $\tau_s^2 = 1$. In particular, $\tau_s$ is an automorphism. To show that $(u_s \tau_s)^3 = 1$, we distinguish the following cases. Let $\alpha_s \neq \alpha \in \Phi_+$. Assume that $o(r_{\alpha_s} r_{\alpha}) < \infty$ and let $R \in \partial^2 \alpha_s \cap \partial^2 \alpha$. Note that we have a homomorphism $N_R \to M_s \to N_s$. By Lemma \ref{Lemma: ustaus V} we have
	\[ \left( \prod\nolimits_{\gamma \in M_{\alpha_s, s\alpha}^G} u_{s\gamma} \right) u_{\alpha} = \left( \prod\nolimits_{\gamma \in M_{\alpha_s, \alpha}^G} \left( \prod\nolimits_{\gamma' \in M_{\alpha_s, s\gamma}^G} u_{\gamma'} \right) u_{s\gamma} \right) \left( \prod\nolimits_{\gamma \in M_{\alpha_s, s\alpha}^G} u_{\gamma} \right) u_{s\alpha} \]
	in $N_R$ and hence $(u_s \tau_s)^3 (u_{\alpha}) = u_{\alpha}$ in $N_s$. Thus we assume $\alpha_s \subsetneq \alpha$. Then we have the following:
	\begin{align*}
		(u_s \tau_s)^3(u_{\alpha}) &= (u_s \tau_s)^2(u_s u_{s\alpha} u_s) \\
		&= (u_s \tau_s u_s)\left( u_s \left( \prod\nolimits_{\gamma \in M_{\alpha_s, \alpha}^G} u_{s\gamma} \right) u_{s\alpha} u_s \right) \\
		&= (u_s \tau_s)\left( \left( \prod\nolimits_{\gamma \in M_{\alpha_s, \alpha}^G} u_{s\gamma} \right) u_{s\alpha} \right) \\
		&= u_s \left( \left( \prod\nolimits_{\gamma \in M_{\alpha_s, \alpha}^G} u_{\gamma} \right) u_{\alpha} \right) \\
		&= u_{\alpha}
	\end{align*}
	Now we assume $-\alpha_s \subseteq \alpha$. Using the previous case, we deduce the following:
	\begin{align*}
		(u_s \tau_s)^3(u_s u_{\alpha} u_s) &= (u_s \tau_s)(u_{s\alpha}) = u_s(u_{\alpha}) = u_s u_{\alpha} u_s \\
		(u_s \tau_s)^3(u_{\alpha}) &= (u_s \tau_s)^2( [u_s, u_{s\alpha}] u_{s\alpha} ) = (u_s \tau_s)^{-1}( [u_s, u_{s\alpha}] u_{s\alpha} ) = u_{\alpha} \qedhere
	\end{align*}
\end{proof}

\begin{definition}\label{Definition: Ps}
	Note that $\phi: \Sym(3) \to \langle u_s, \tau_s \rangle \leq \Aut(N_s), \begin{cases*}
		\begin{pmatrix}
			1 & 2
		\end{pmatrix} \mapsto u_s \\
		\begin{pmatrix}
			2 & 3
		\end{pmatrix} \mapsto \tau_s
	\end{cases*}$ is an epimorphism. Thus we define the group $P_s := \Sym(3) \ltimes_{\phi} N_s$. For short we will denote the elements in $\Sym(3)$ by their images in $\Aut(N_s)$. Note that $\tau_s n_s \tau_s = \tau_s(n_s) \in N_s$. In particular, we have $\tau_s u_{\alpha} \tau_s = u_{s\alpha}$ for each $\alpha_s \neq \alpha \in \Phi_+$. Note that $U_+ \cong \langle u_s \rangle \ltimes N_s \leq P_s$.
\end{definition}

\section{An action of the groups $P_s$}\label{Section: Action of Ps}

Recall, that $\mathcal{M}$ is a faithful and Weyl-invariant commutator blueprint of type $(W, S)$. In this section we will show that the groups $P_s$ act faithfully on a chamber system $\Cbf$ over $S$ for every $s\in S$. Moreover, we will show that the braid relations $(\tau_s \tau_t)^{m_{st}}$ act trivial on $\Cbf$. In particular, the action of the groups $P_s$ extend to an action of $G$ on $\Cbf$. We use this action in Theorem \ref{RGDsystem} to construct an RGD-system containing $U_+$ as a subgroup.

\begin{definition}
	We let $U_{1_W} := \{ 1 \} \leq U_+$. The set of chambers is given by $\mathcal{C} := \{ gU_w \mid g\in U_+, w\in W \}$, and $s$-adjacency is defined as follows:
	\[ gU_w \sim_s hU_{w'} :\Leftrightarrow w' \in \{w, ws\} \text{ and } g^{-1}h \in U_w \cup U_{ws} \]
	Then $\Cbf = (\C, (\sim_s)_{s\in S})$ is a chamber system over $S$.
\end{definition}
The idea of considering this chamber system is not new (cf.\ \cite[Section $8.7$]{AB08}). Before we define an action of $P_s$ on the chamber system $\Cbf$ we note that every element of $U_+$ can be written uniquely as $nu$ with $n \in N_s$ and $u \in U_s$ by Lemma \ref{UplusNsUs}. Thus it suffices to define the action on cosets $nuU_w$ with $n\in N_s, u \in U_s$ and $w\in W$. To show that our assignment will actually be an action we need the following auxiliary result.

\begin{lemma}\label{NsUw}
	For $n\in N_s$ the following hold:
	\begin{enumerate}[label=(\alph*)]
		\item If $n\in U_w$, then $n^{\tau_s} \in N_s \cap U_{sw}$;
		
		\item If $\ell(sw) = \ell(w) +1$ and $n^{u_s} \in U_w$, then $n^{\tau_s u_s} \in N_s \cap U_w$.
	\end{enumerate}
\end{lemma}
\begin{proof}
	Let $w\in W$, let $G = (c_0, \ldots, c_k) \in \mathrm{Min}_s(w)$ and let $(\alpha_1, \ldots, \alpha_k)$ be the sequence of roots crossed by $G$. Since $n\in U_w$, there exists $u_i \in U_{\alpha_i}$ such that $n = u_1 \cdots u_k$. If $\ell(sw) = \ell(w) +1$, then $u_i^{\tau_s} \in U_{s\alpha_i} \leq U_{sw}$ and hence $n^{\tau_s} \in U_{sw}$. Thus we assume that $\ell(sw) = \ell(w) -1$ and hence $\alpha_1 = \alpha_s$. Since $U_{\alpha_i} \leq N_s$ for each $2 \leq i \leq k$, we have $u_1 = n (u_2 \cdots u_k)^{-1} \in N_s \cap U_s = \{1\}$. Thus $n^{\tau_s} \in U_{sw}$ and Assertion $(a)$ follows. Now we assume that $\ell(sw) = \ell(w) +1$ and that $n^{u_s} \in U_w$. Note that $n^{u_s} \in N_s$. Then $(a)$ provides $n^{u_s \tau_s} \in N_s \cap U_{sw}$. Since $\ell(ssw) = \ell(w) = \ell(sw) -1$, we have $u_s \in U_{sw}$ and hence $n^{u_s \tau_s u_s} \in N_s \cap U_{sw}$. Using Corollary \ref{ustaus} and Assertion $(a)$, we obtain $n^{\tau_s u_s} = n^{u_s \tau_s u_s \tau_s} \in N_s \cap U_w$.
\end{proof}

\begin{remark}\label{Remark: Presentation of Ps}
	Let $\langle G_s \mid R_s \rangle$ be a presentation of $N_s$. Then a presentation of $P_s$ is given by $\langle u_s, \tau_s, G_s \mid u_s^2, \tau_s^2, (u_s \tau_s)^3, R_s, u_s n u_s = n^{u_s}, \tau_s n \tau_s = n^{\tau_s} \text{ for each } n \in G_s \rangle$.
\end{remark}

\begin{proposition}\label{Proposition: Ps acts on C}
	For $s\in S$ the group $P_s$ acts on $\Cbf$ as follows:
	\[ g.nuU_w := \begin{cases}
	gnu U_w & g \in U_+ \\
	n^{\tau_s} U_{sw} & g = \tau_s, \ell(sw) = \ell(w) -1 \text{ or } u=1 \\
	n^{\tau_s} u_s U_w & g = \tau_s, \ell(sw) = \ell(w) +1, u = u_s
	\end{cases} \]
	Moreover, this action is faithful.
\end{proposition}
\begin{proof}
	For $g\in U_+ \cup \{ \tau_s \}$ we let $\phi_g: \C \to \C, nuU_w \mapsto g.nuU_w$.
	
	\medskip
	\noindent \emph{The mapping $\phi_g$ is well-defined:} We note that $u_s.nuU_w = u_snuU_w = n^{u_s}u_suU_w$. We first show that the assignment is well-defined. Since $\phi_g$ for $g\in U_+$ is given by multiplication from the left, it suffices to consider $\phi_{\tau_s}$. Suppose $w\in W$ and $n, n' \in N_s, u, u' \in U_s$ such that $nuU_w = n'u'U_w$. Then $u^{-1}n^{-1}n'u' \in U_w$.
	\begin{enumerate}[label=(\alph*)]
		\item $\ell(sw) = \ell(w)-1$: Then $u_s \in U_w$ and hence $n^{-1}n' \in U_w$. Using Lemma \ref{NsUw}$(a)$, we obtain $(n^{-1}n')^{\tau_s} \in U_{sw}$. This implies $\tau_s.nuU_w = n^{\tau_s}U_{sw} = (n')^{\tau_s}U_{sw} = \tau_s.n'u'U_w$.
		
		\item $\ell(sw) = \ell(w)+1$: We distinguish the following three cases:
		\begin{itemize}
			\item $u = 1 = u'$: Then the claim follows as in $(a)$.
			
			\item $\{u, u'\} = \{ 1, u_s \}$: Assume $u \neq 1 = u'$. Then we have $u^{-1}n^{-1}n' \in U_w$. Since $\ell(sw) = \ell(w) +1$, we have $U_w \leq N_s$ and hence $u_s = u^{-1} \in N_s$. This is a contradiction. The case $u = 1 \neq u'$ is similar.
			
			\item $u = u_s = u'$: Then $(n^{-1}n')^{u_s} \in N_s \cap U_w$. Using Lemma \ref{NsUw}$(b)$, we obtain $(n^{-1} n')^{\tau_s u_s} \in N_s \cap U_w$ and hence $\tau_s.nuU_w = n^{\tau_s}u_sU_w = (n')^{\tau_s}u_sU_w = \tau_s.n'u'U_w$.
		\end{itemize}
	\end{enumerate}
	Thus $\phi_g$ is well-defined.
	
	\medskip
	\noindent \emph{$\phi_g$ is bijective for each $g\in U_+ \cup \{ \tau_s \}$:} We will show that $\phi_{g^{-1}} \circ \phi_g = \id$. If $g\in U_+$ there is nothing to show. Thus we consider $g = \tau_s$. By construction and Corollary \ref{ustaus} we have $\phi_{\tau_s} \circ \phi_{\tau_s} = \id$ and $\phi_g$ is bijective for every $g\in U_+ \cup \{ \tau_s \}$.
	
	\medskip
	\noindent \emph{$\phi_g \in \Aut(\Cbf)$:} As $\phi_g$ is bijective, it suffices to show that $\phi_g$ preserves $t$-adjacency for each $t\in S$. Suppose $n, n' \in N_s, u, u' \in U_s$ and $w, w' \in W$ such that $nuU_w \sim_t n'u'U_{w'}$. Then we have $w' \in \{ w, wt \}$ and $u^{-1}n^{-1}n'u' \in U_w \cup U_{wt}$. Since for $g\in U_+$ the bijection $\phi_g$ is multiplication with $g$ from the left, it preserves $t$-adjacency and it suffices to consider $\phi_{\tau_s}$. We distinguish the following cases:
	\begin{enumerate}[label=(\alph*)]
		\item $u = 1 = u'$: Then $\tau_s. nU_w = n^{\tau_s} U_{sw}$ and $\tau_s.n' U_{w'} = (n')^{\tau_s} U_{sw'}$. Because of the $t$-adjacency we have $n^{-1}n \in U_w \cup U_{wt}$ and Lemma \ref{NsUw}$(a)$ implies $(n^{-1})^{\tau_s} (n')^{\tau_s} = (n^{-1}n')^{\tau_s} \in U_{sw} \cup U_{swt}$. Since $sw' \in \{ sw, swt \}$, we deduce $\phi_{\tau_s}(nU_w) \sim_t \phi_{\tau_s}(n'U_{w'})$.
		
		\item $\ell(sw) = \ell(w) -1$ and $\ell(sw') = \ell(w') -1$: Then $nuU_w = nU_w$ and $n'u'U_{w'} = n'U_{w'}$ and the claim follows from $(a)$.
		
		\item $\ell(sw) = \ell(w) +1$ and $\ell(sw') = \ell(w') +1$: Recall that $w' \in \{ w, wt \}$. If $u = 1 = u'$ the claim follows from $(a)$. If $u = u_s = u'$ we have $(n^{-1} n')^{u_s} \in U_w \cup U_{wt}$ and $\tau_s.nu_sU_w = n^{\tau_s}u_s U_w, \tau_s.n' u_s U_{w'} = (n')^{\tau_s} u_s U_{w'}$. If $\ell(swt) = \ell(wt) +1$, then we have $(n^{-1} n')^{\tau_s u_s} \in N_s \cap \left( U_{w} \cup U_{wt} \right)$ by Lemma \ref{NsUw}$(b)$ and we deduce $\phi_{\tau_s}(nuU_w) \sim_t \phi_{\tau_s}(n'u'U_{w'})$. Thus we assume $\ell(swt) = \ell(wt) -1$. Then $u_s \in U_{wt}$. Since $\ell(wt)-1 = \ell(swt) \geq \ell(sw) -1 = \ell(w)$, we have $\ell(wt) = \ell(w) +1$ and thus $(n^{-1} n')^{u_s} \in U_w \cup U_{wt} = U_{wt}$. This implies $n^{-1}n' \in U_{wt}$. By Lemma \ref{conditionF} we infer $swt = w$. Now Lemma \ref{NsUw}$(a)$ yields $(n^{-1} n')^{\tau_s} \in N_s \cap U_{swt} = N_s \cap U_w \leq N_s \cap U_{wt}$ and, as $u_s \in U_{wt}$, $(n^{-1} n')^{\tau_s u_s} \in U_{wt}$. In particular, $\phi_{\tau_s}(nuU_w) \sim_t \phi_{\tau_s}(n'u'U_{w'})$.
		
		If $u=1\neq u'$ we have $(n^{-1})n' u_s \in U_w \cup U_{wt}$ and $\tau_s.nU_w = n^{\tau_s} U_{sw}, \tau_s.n'u_sU_{w'} = (n')^{\tau_s} u_s U_{w'}$. If $\ell(swt) = \ell(wt) +1$, we would have $U_w, U_{wt} \leq N_s$ and hence $u_s \in N_s$. Thus we have $\ell(swt) = \ell(wt) -1$. Since $\ell(sw') = \ell(w') +1$ and $w' \in \{w, wt\}$, we deduce $w = w'$. As $\ell(sw) = \ell(w)+1$, we obtain $\ell(wt) -1 = \ell(swt) \geq \ell(sw) -1 = \ell(w)$. This yields $\ell(wt) = \ell(w) +1$ and hence $swt = w$ as before. This implies $w' = w = swt \in \{ sw, swt \}$ and $U_w \leq U_{wt}$. Thus we obtain $(n^{-1})n' u_s \in U_{wt}$ and hence $(n^{-1})n' \in U_{wt}$. Using Lemma \ref{NsUw}$(a)$ we obtain $(n^{-1} n')^{\tau_s} \in U_{swt} \leq U_{sw}$ (since $\ell(swt) = \ell(sw)-1$). This implies $(n^{-1} n')^{\tau_s} u_s \in U_{sw}$ and hence $\phi_{\tau_s}(nU_w) \sim_t \phi_{\tau_s}(n'u'U_{w'})$. The case $u \neq 1 = u'$ is similar.
		
		\item Without loss of generality we assume $\ell(sw) = \ell(w) -1$ and $\ell(sw') = \ell(w') +1$. This implies $w \neq w'$ and hence $w' = wt$. Thus $\ell(wt) = \ell(w') = \ell(sw') -1 \leq \ell(sw) = \ell(w) -1$ and hence $\ell(wt) = \ell(w) -1$. Since $\ell(swt) = \ell(w)$, Lemma \ref{conditionF} implies $w = swt$.
		
		Now we have $nuU_w = nU_w$ and $\tau_s.nU_w = n^{\tau_s}U_{sw}$. If $u' = 1$, the claim follows from $(a)$. Thus we assume $u' = u_s$. Then $\tau_s.n'u_sU_{w'} = (n')^{\tau_s}u_s U_{w'}$. Since $w' = wt = sw \in \{ sw, swt \}$ it suffices to show that $(n^{-1} n')^{\tau_s} u_s \in U_{sw} \cup U_{swt}$. As $\ell(wt) = \ell(w) -1$, we have $U_{wt} \leq U_w$. Because $\ell(sw) = \ell(w) -1$ and $n^{-1} n' u_s \in U_w \cup U_{wt} = U_w$ we have $u_s \in U_w$ and hence $n^{-1} n' \in U_w$. Using Lemma \ref{NsUw}$(a)$ we deduce $(n^{-1} n')^{\tau_s} \in U_{sw}$. Since $\ell(swt) = \ell(w) = \ell(sw) +1$, we obtain $U_{sw} \leq U_{swt}$. This implies $(n^{-1} n')^{\tau_s} u_s \in U_{swt}$ and we obtain $\phi_{\tau_s}(nU_w) \sim_t \phi_{\tau_s}(n'u'U_{w'})$.
	\end{enumerate}
	
	\medskip
	\noindent \emph{The assignment $g \mapsto \phi_g$ for $g\in U_+ \cup \{\tau_s\}$ extends to a homomorphism $P_s \to \Aut(\Cbf)$:} For this we need to consider a presentation of $P_s$ (cf.\ Remark \ref{Remark: Presentation of Ps}) and show that every relation of $P_s$ acts trivial on the chamber system $\Cbf$. Since the action of $U_+ \leq P_s$ is via multiplication from the left it suffices to consider relations concerning $\tau_s$. As we have already seen before, $\tau_s^2$ acts trivial. Suppose $m, m' \in N_s$ with $\tau_s m \tau_s = \tau_s(m) = (m')^{-1}$. Then
	\[ \tau_s m \tau_s m'.nuU_w = \tau_sm. (m'n)^{\tau_s} (\tau_s.uU_w) = (m(m'n)^{\tau_s})^{\tau_s} uU_w = m^{\tau_s}m'nuU_w = nuU_w \]
	Thus it suffices to show that $(u_s \tau_s)^3$ acts trivial on $\Cbf$. As $(u_s \tau_s)^3. nuU_w = n^{(\tau_s u_s)^3} \cdot (u_s \tau_s)^3. uU_w$, we can assume that $n=1$, since $(u_s \tau_s)^3$ acts trivial on $N_s$ by Corollary \ref{ustaus}. If $\ell(sw) = \ell(w) -1$, then $uU_w = U_w = u_s U_w$ and we obtain the following:
	\[ (u_s \tau_s)^3.uU_w = (u_s \tau_s)^2.u_sU_{sw} = u_s \tau_s. U_{sw} = u_sU_w = U_w \]
	Thus we can assume that $\ell(sw) = \ell(w) +1$. We distinguish the cases $u = 1$ and $u = u_s$:
	\begin{align*}
		(u_s \tau_s)^3.U_w &= (u_s \tau_s)^2. U_{sw} = u_s \tau_s. u_s U_w = U_w \\
		(u_s \tau_s)^3. u_s U_w &= (u_s \tau_s)^2. U_w = u_s \tau_s. U_{sw} = u_s U_w
	\end{align*}
	
	\medskip
	\noindent \emph{The homomorphism $P_s \to \Aut(\Cbf)$ is injective:} We have to show that each $1 \neq g\in P_s$ induces a non-trivial automorphism of the chamber system. We first consider $1 \neq g\in \Sym(3) = \{ 1, u_s, u_s \tau_s, u_s \tau_s u_s, \tau_s u_s, \tau_s \}$. Then we have the following:
	\allowdisplaybreaks
	\begin{align*}
		u_s. U_{1_W} = u_s U_{1_W}, \,\, u_s \tau_s. U_{1_W} = U_s, \,\, u_s \tau_s u_s. U_s = u_s U_{1_W}, \,\, \tau_s u_s. U_s = U_{1_W}, \,\, \tau_s. U_{1_W} = U_s
\end{align*}
	Thus each $1 \neq g\in \Sym(3)$ acts non-trivial. Now we consider the general case. Let $1 \neq g\in P_s$. Then there exist $x\in \Sym(3), n\in N_s$ such that $g = xn$. If $x=1$, we have $g.n^{-1}U_{1_W} = U_{1_W} \neq n^{-1} U_{1_W}$. Otherwise the let $c \in \C$ be as above such that $x.c \neq c$. Then $g.n^{-1}c \neq n^{-1} c$ and the claim follows.
\end{proof}

\begin{theorem}\label{Theorem: braid relations act trivial}
	We have $(\tau_s \tau_t)^{m_{st}} = \id \in \Aut(\Cbf)$ for all $s, t\in S$ with $m_{st} < \infty$.
\end{theorem}
\begin{proof}
	We first introduce some notation. For $J \subseteq S$ we define $\Phi^J := \{ w\alpha_s \mid s\in J, w\in \langle J \rangle \}$ and $\Phi_+^J := \Phi^J \cap \Phi_+$. Moreover, we define for all $s \neq t\in S$ the subgroup $U_{s, t} := \langle U_{\alpha} \mid \alpha \in \Phi_+^{\{s, t\}} \rangle$ and $N_{s, t} := \langle x^{-1} U_{\alpha} x \mid x \in U_{s, t}, \alpha \in \Phi_+ \backslash \Phi_+^{\{s, t\}} \rangle$. Then $N_{s, t}$ is a normal subgroup of $U_+$ which is stabilized by $\tau_s$ and by $\tau_t$.
	
	\medskip
	\noindent \emph{Step 1: We have $[ (\tau_s \tau_t)^{m_{st}}, n ] = 1$ in the free product with amalgamation $P_s *_{U_+} P_t$ for all $s\neq t \in S$ with $m_{st} < \infty$ and all $n\in N_{s, t}$.} The verification is technical but straight forward. For the proof we refer to Lemma \ref{Lemma: mst=2}, \ref{Lemma: mst=3}, \ref{Lemma: mst=4} and \ref{Lemma: mst=6} in the appendix.
	
	\medskip
	\noindent \emph{Step 2: The rank $2$ residues of $\Cbf$ are spherical buildings.} Suppose $s \neq t\in S$ with $m_{st} < \infty$ and let $J := \{ s, t \}$. Since $\mathcal{M}$ is faithful, the mapping $U_{r_J} \to U_+$ is injective. Considering the sub-chamber system $\Cbf_J = (\C_J, (\sim_j)_{j\in J} )$ with $\C_J = \{ uU_w \mid u \in U_{s, t}, w\in \langle J \rangle \}$. This is exactly the chamber system which we get from the RGD-system over $\FF_2$ of type $I_2(m_{st})$. This chamber system is a building by \cite[Exercise $8.36(b)$]{AB08}.
	
	\medskip
	\noindent \emph{Step 3: For $s\neq t\in S$ with $m_{st} < \infty$ we have $(\tau_s \tau_t)^{m_{st}} = \id$.} We put $J := \{s, t\}$.	For $w\in W$ we let $w' \in W, w_J \in \langle J \rangle$ be such that $w = w_J w'$ and $\ell(sw') = \ell(w') +1 = \ell(tw')$. Then the action of $\tau_s$ on $uU_w$ only depends on $u$ and $w_J$ and is independent on $w'$, i.e.\ for $u, u' \in U_{s, t}$ and $w_J' \in \langle J \rangle$ with $\tau_s.uU_{w_J} = u' U_{w_J'}$, we have $\tau_s.uU_w = u' U_{w_J'w'}$. Thus it suffices to show the claim for $w\in \langle J \rangle$. We restrict the action of $(\tau_s \tau_t)^{m_{st}}$ on $\Cbf$ to the chambers of the form $uU_w$ with $u\in U_{s, t}$ and $w\in \langle J \rangle$. 
	
	Restricting $\tau_s, \tau_t$ to the sub-chamber system $\Cbf_J$, we infer that $(\tau_s \tau_t)^{m_{st}}$ is an automorphism of $\Cbf_J$. By the previous lemma this chamber system is a building of type $(\langle J \rangle, J)$. Since $(\tau_s \tau_t)^{m_{st}}$ fixes all chambers $U_w$ with $w\in \langle J \rangle$, it fixes the two opposite chambers $U_{1_W}$ and $U_{r_{J}}$. Since every panel contains exactly three chambers, the automorphism $(\tau_s \tau_t)^{m_{st}}$ fixes $R_{\{s\}}(U_{1_W})$ and $R_{\{t\}}(U_{1_W})$. Using Theorem \ref{Theorem5.205AB08}, we obtain $(\tau_s \tau_t)^{m_{st}}.uU_w = uU_w$ for all $u\in U_{s, t}$ and $w\in \langle J \rangle$. This finishes the claim.
\end{proof}

\subsection*{The RGD-system}

\begin{definition}\label{Definition: direct limit G}
	We denote the direct limit of the groups $U_+, (P_s)_{s\in S}, (\langle \tau_s \rangle)_{s\in S}, W$ with canonical inclusions $U_+ \hookrightarrow P_s, \langle \tau_s \rangle \hookrightarrow P_s, \langle \tau_s \rangle \hookrightarrow W, \tau_s \mapsto s$ by $G$
\end{definition}

\begin{lemma}
	Let $s_1, \ldots, s_n, t_1, \ldots, t_m, s, t \in S$ be such that $s_1 \cdots s_n \alpha_s = t_1 \cdots t_m \alpha_t$. Then we have $U_{\alpha_s}^{\tau_n \cdots \tau_1} = U_{\alpha_t}^{\tau_m' \cdots \tau_1'}$ in $G$, where $\tau_i = \tau_{s_i}$ and $\tau_j' = \tau_{t_j}$.
\end{lemma}
\begin{proof}
	The claim follows if $U_{\alpha_s}^{\tau_n \cdots \tau_1 \tau_1' \cdots \tau_m'} = U_{\alpha_t}$. Suppose $f_1, \ldots, f_k \in S$ with $\ell(f_1 \cdots f_k) = k$ and $f_1 \cdots f_k = t_m \cdots t_1 s_1 \cdots s_n$. Then $f_k \cdots f_1 = s_n \cdots s_1 t_1 \cdots t_m$ and since every relation in $W$ is a relation in $G$, we obtain $\tau_{f_k} \cdots \tau_{f_1} = \tau_{s_n} \cdots \tau_{s_1} \tau_{t_1} \cdots \tau_{t_m}$.
	
	Let $i := \max\{ 1, \ldots, k \mid \exists r\in S: f_i \cdots f_k \alpha_s = \alpha_r \}$. For $g := f_1 \cdots f_k$ we have $g \alpha_s = \alpha_t$ and hence $g^{-1} \in \alpha_s$. This implies $\ell(gs) = \ell((gs)^{-1}) = \ell(sg^{-1}) >\ell(g^{-1}) = \ell(g)$. This implies $f_k \neq s$ and hence $f_k\alpha_s \in \Phi_+$. Thus the roots $\alpha_s, f_k \alpha_s, \ldots, f_i \cdots f_k \alpha_s = \alpha_r$ are all positive roots and we obtain $U_{\alpha_s}^{\tau_{f_k} \cdots \tau_{f_i}} = U_{f_i \cdots f_k \alpha_s} = U_{\alpha_r}$ in $G$. If $i=1$ we are done. Otherwise we repeat the argument with $g := f_1 \cdots f_{i-1}$. After finitely many steps we are done.
\end{proof}

\begin{definition}
	Let $\alpha \in \Phi$ be a root. Then there exist $w\in W$ and $s\in S$ with $\alpha = w\alpha_s$. Let $s_1, \ldots, s_k \in S$ be such that $w = s_1 \cdots s_k$ and let $\tau_i := \tau_{s_i}$. Then we define
	\[ U_{\alpha} := U_{\alpha_s}^{\tau_k \cdots \tau_1} \]
	In view of the previous lemma, the group $U_{\alpha}$ is well-defined. Moreover, we let $\mathcal{D}_{\mathcal{M}} := (G, (U_{\alpha})_{\alpha \in \Phi})$.
\end{definition}

\begin{theorem}\label{RGDsystem}
	$\mathcal{D}_{\mathcal{M}}$ is an RGD-system and $\mathcal{M}$ is integrable.
\end{theorem}
\begin{proof}
	It is a consequence of Proposition \ref{Proposition: Ps acts on C} and Theorem \ref{Theorem: braid relations act trivial} that for each $s\in S$ the homomorphism $P_s \to G$ is injective. We consider the different axioms:
	\begin{enumerate}[label=(RGD\arabic*)] \setcounter{enumi}{-1}
		\item The mappings $P_s \to G$ are injective and hence the groups $U_{\alpha}$ are non-trivial.
		
		\item Let $\{ \alpha, \beta \} \subseteq \Phi$ be a prenilpotent pair with $\alpha \neq \beta$. Then there exists $w\in W$ such that $\{ w\alpha, w\beta \} \in \P$. By definition of the root groups and the commutator blueprint we deduce ($\tau_w$ is a product of suitable $\tau_s$)
		\begin{align*}
			[ U_{\alpha}, U_{\beta} ] = [U_{w\alpha}, U_{w\beta}]^{\tau_w} &\leq \langle U_{\gamma} \mid \gamma \in (w\alpha, w\beta) \rangle^{\tau_w} \\
			&= \langle U_{w^{-1}\gamma} \mid \gamma \in (w\alpha, w\beta) \rangle \\
			&= \langle U_{\gamma} \mid \gamma \in (\alpha, \beta) \rangle
		\end{align*}
		
		\item For $s\in S$ we have $(u_s \tau_s)^3 = 1$ and hence $\tau_s = \tau_s (u_s \tau_s)^3 = u_{-s} u_s u_{-s}$ by Corollary \ref{ustaus}. Let $\alpha \in \Phi$ be a root. Then there exist $w\in W, t\in S$ such that $\alpha = w\alpha_t$. Let $s_1, \ldots, s_k \in S$ be such that $w = s_1 \cdots s_k$ and let $\tau_i := \tau_{s_i}$. Then $s\alpha = ss_1 \cdots s_k\alpha_t$ and we deduce
		\[ U_{\alpha}^{\tau_s} = \left( U_{\alpha_t}^{\tau_k \cdots \tau_1} \right)^{\tau_s} = U_{\alpha_t}^{\tau_k \cdots \tau_1 \tau_s} = U_{s\alpha} \]
		
		\item Since $P_s \to G$ is injective, we have $\tau_s \notin U_+$. As $U_+^{u_s} = U_+$ and $(u_s \tau_s)^3 = 1$, we infer $u_{-s} = \tau_s u_s \tau_s = u_s \tau_s u_s \notin U_+^{u_s} = U_+$
		
		\item Since $G$ is generated by $U_{\alpha}$ and $\tau_s$, it is generated by all root groups.
	\end{enumerate}
	Note that $\mathcal{M}_{\mathcal{D}}$ is a commutator blueprint of type $(W, S)$. By definition we have $M_{\alpha, \beta}^G = M(\mathcal{D})_{\alpha, \beta}^G$ for each $(G, \alpha, \beta) \in \mathcal{I}$. We deduce that $\mathcal{M}$ is integrable.
\end{proof}

\begin{remark}\label{Remark: FPRS}
	In \cite{CR09}, Caprace and Rémy have introduced property (FPRS) for RGD-systems and we refer to loc.\ cit.\ for the precise definition. It is mentioned in \cite[Remark before Lemma $5$]{CR09} that M\"uhlherr announced the construction of an example of an RGD-system of right-angled type and of rank $3$ which does not satisfy property (FPRS). He informed the author that this construction is not available in form of a preprint.
\end{remark}

\begin{corollary}\label{Corollary: 2^N examples}
	Assume that $(W, S)$ is a universal Coxeter system of rank $n \geq 2$. Then there exists an RGD-system of type $(W, S)$ which does not satisfy property (FPRS).
\end{corollary}
\begin{proof}
	Let $s\neq t \in S$ and for every $n \in \NN$ we let $J_n \subseteq \{ 1, \ldots, n \}$ with $1 \in J_n$. We consider the Weyl-invariant commutator blueprint $\mathcal{M}(\NN, (J_n)_{n\in \NN}, s, t)$ from \cite[Theorem $4.6$]{BiConstruction}. As mentioned in the introduction, it is also faithful and Theorem \ref{RGDsystem} implies that the commutator blueprint is integrable. We let $\mathcal{D} = (G, (U_{\alpha})_{\alpha \in \Phi})$ be the RGD-system associated with the commutator blueprint.
	
	In this proof we adopt the notation of \cite[$2.1$]{CR09}. Let $H_n = (c_0, \ldots, c_n) \in \mathrm{Min}(w)$ be of type $(s, t, s, t, \ldots)$ with $\ell(w) = n$ (i.e.\ $H_3$ has type $(s, t, s)$) and define $\alpha_n \in \Phi_+$ to be the root containing $c_{n-1}$ but not $c_n$. Then $\lim_{i \to \infty} \ell(1_W, -\alpha_{2i-1}) = \infty$. Assume $\mathcal{D}$ satisfies (FPRS). Then there exists $n_0 \in \NN$ such that for all $i \geq n_0$ the root group $U_{\alpha_{2i-1}}$ fixes the ball $B(c_+, 2)$ pointwise. But then $[u_{\alpha_1}, u_{\alpha_{2i-1}}] = \prod_{j \in J_i} u_{\alpha_{2j}}$ would also fix $B(c_+, 2)$ pointwise, which is a contradiction, as $1 \in J_i$.
\end{proof}

\appendix
\section{The braid relations act trivial}\label{Appendix}

We adopt the notation from Theorem \ref{Theorem: braid relations act trivial}: For $J \subseteq S$ we define $\Phi^J := \{ w\alpha_s \mid s\in J, w\in \langle J \rangle \}$ and $\Phi_+^J := \Phi^J \cap \Phi_+$. Moreover, we define for all $s \neq t\in S$ the subgroup $U_{s, t} := \langle U_{\alpha} \mid \alpha \in \Phi_+^{\{s, t\}} \rangle$ and $N_{s, t} := \langle x^{-1} U_{\alpha} x \mid x \in U_{s, t}, \alpha \in \Phi_+ \backslash \Phi_+^{\{s, t\}} \rangle$. Then $N_{s, t}$ is a normal subgroup of $U_+$ which is stabilized by $\tau_s$ and by $\tau_t$. Note that any element in $U_{s, t}$ has a unique expression: Let $G \in \mathrm{Min}_s(r_{\{s, t\}})$ and let $(\alpha_1, \ldots, \alpha_m)$, $m = m_{st}$ be the sequence of roots crossed by $G$. Then any element in $U_{s, t}$ can be written as $\prod_{i=1}^m u_{\alpha_i}^{\epsilon_i}$ with $\epsilon_i \in \{0, 1\}$.

In the appendix we show that $[(\tau_s \tau_t)^{m_{st}}, n] = 1$ holds in $P_s *_{U_+} P_t$ for all $s\neq t \in S$ with $m_{st} < \infty$ and all $n\in N_{s, t}$. It suffices to consider a generating set of $N_{s, t}$, i.e.\ $n \in \{ u^{-1} u_{\alpha} u \mid u \in U_{s, t}, \alpha \in \Phi_+ \backslash \Phi_+^{\{s, t\}} \}$. We abbreviate $u_{ws} := u_{w\alpha_s} \in U_{ws} \backslash \{1\}$, i.e.\ $u_{ts} = u_{t\alpha_s}$. We will always assume that $-\beta \subseteq \alpha$, if $u_{\beta}$ appears in $u$. Otherwise we can reduce $u$ as we see in the next example.

\begin{example}
	Suppose $\alpha \in \Phi_+ \backslash \Phi_+^{\{s, t\}}$ with $-\alpha_s \not\subseteq \alpha$. Then $\{ \alpha_s, \alpha \} \in \P$ by definition and we have $u_s u_{\alpha} u_s = \left( \prod_{\gamma \in M_{\alpha_s, \alpha}^G} u_{\gamma} \right) u_{\alpha}$ for some $G \in \mathrm{Min}$ with $\alpha_s, \beta \in \Phi(G)$.
\end{example}

\begin{remark}
	Let $\alpha, \beta, \gamma \in \Phi_+$ be pairwise distinct and pairwise prenilpotent and suppose $\alpha \subseteq \gamma$. Then for $\delta \in (\alpha, \beta)$ we have $\emptyset \neq (-\beta) \cap (-\gamma) \subseteq (-\beta) \cap (-\alpha) \cap (-\gamma) \subseteq (-\delta) \cap (-\gamma)$ and hence $\{ \delta, \gamma \}$ is a prenilpotent pair. This observation will be useful in the following lemmas.
\end{remark}

\begin{convention}
	For short we will write $u_s.n := u_s n u_s$ and $\tau_s.n := \tau_s n \tau_s = \tau_s(n)$.
\end{convention}

\begin{lemma}\label{Lemma: mst=2}
	Suppose $m_{st} = 2$. Then $[(\tau_s \tau_t)^2, n] = 1$ holds for all $n \in N_{s,t}$ in $P_s *_{U_+} P_t$.
\end{lemma}
\begin{proof}
	It suffices to show that $\tau_s \tau_t.n = \tau_t \tau_s.n$. If $n = u_{\alpha}$, then the claim follows. We will argue by induction: Let $u = u_t u'$ for some $u, u' \in U_{s, t}$. Assume that the braid relation acts trivial on $u'. u_{\alpha} = u' u_{\alpha} (u')^{-1}$ and on $u_t \tau_t. \left( u'. u_{\alpha} \right)$. Then we compute the following:
	\allowdisplaybreaks
	\begin{align*}
		\tau_t \tau_s. \left( u_t u'. u_{\alpha} \right) &= \tau_t \tau_s u_t. \left( u'. u_{\alpha} \right) \\
		&= \tau_t u_t \tau_s. \left( u'. u_{\alpha} \right) \\
		&= \tau_t u_t \tau_t \tau_s \tau_t. \left( u'. u_{\alpha} \right) \\
		&= u_t \tau_t u_t \tau_s \tau_t. \left( u'. u_{\alpha} \right) \\
		&= u_t \tau_t \tau_s. \left( u_t \tau_t. \left( u'. u_{\alpha} \right) \right) \\
		&= u_t \tau_s \tau_t. \left( u_t \tau_t. \left( u'. u_{\alpha} \right) \right) \\
		&= \tau_s u_t \tau_t. \left( u_t \tau_t. \left( u'. u_{\alpha} \right) \right) \\
		&= \tau_s \tau_t. \left( u_t u'. u_{\alpha} \right)
	\end{align*}
	We consider the following cases:
	\allowdisplaybreaks
	\begin{itemize}
		\item $u = u_r$ for $r\in \{s, t\}$: Then $u' = 1$ and $u_r \tau_r. u_{\alpha} = u_{r\alpha} [u_{r\alpha}, u_r]$. Writing $[u_{r\alpha}, u_r] = u_{\gamma_1} \cdots u_{\gamma_k}$, we obtain $\gamma_i \in \Phi_+ \backslash \Phi^{\{s, t\}}$ and the braid relation acts trivial on each $u_{\gamma_i}$.
		
		\item $u = u_t u_s$: Then $u' = u_s$ and $u_t \tau_t. \left( u_s. u_{\alpha} \right) = u_t u_s. u_{t\alpha} = u_s. u_{t\alpha} [u_{t\alpha}, u_t]$. Again, writing $[u_{t\alpha}, u_t] = u_{\gamma_1} \cdots u_{\gamma_k}$ with $\gamma_i \in \Phi_+ \backslash \Phi^{\{s, t\}}$, we either have $\{ \alpha_s, \gamma_i \} \in \P$ or $-\alpha_s \subseteq \gamma_i$. We deduce $u_s. u_{t\alpha} [u_{t\alpha}, u_t] = u_s u_{t\alpha} u_s \prod_{i=1}^k u_s u_{\gamma_i} u_s$. In both cases we already know that the braid relation acts trivial on $u_s u_{\gamma_i} u_s$ and the claim follows. \qedhere
	\end{itemize}
\end{proof}

\begin{lemma}\label{Lemma: mst=3}
	Suppose $m_{st} = 3$. Then $[(\tau_s \tau_t)^3, n] = 1$ holds for all $n \in N_{s,t}$ in $P_s *_{U_+} P_t$.
\end{lemma}
\begin{proof}
	It suffices to show that $\tau_s \tau_t \tau_s.n = \tau_t \tau_s \tau_t.n$. If $n = u_{\alpha}$, then the claim follows. As in the case $m_{st} = 2$ we let $u = u_t u'$ for some $u, u' \in U_{s, t}$ and assume that the braid relation acts trivial on $u'. u_{\alpha}$ and on $u_t \tau_t. \left( u'. u_{\alpha} \right)$. Then we compute the following:
	\allowdisplaybreaks
	\begin{align*}
		\tau_s \tau_t \tau_s. \left( u_t u'. u_{\alpha} \right) &= \tau_s u_s \tau_t \tau_s. \left( u'. u_{\alpha} \right) \\
		&= \tau_s u_s \tau_s \tau_t \tau_s \tau_t. \left( u'. u_{\alpha} \right) \\
		&= u_s \tau_s u_s \tau_t \tau_s \tau_t. \left( u'. u_{\alpha} \right) \\
		&= u_s \tau_s \tau_t \tau_s. \left( u_t \tau_t. \left( u'. u_{\alpha} \right) \right) \\
		&= u_s \tau_t \tau_s \tau_t. \left( u_t \tau_t. \left( u'. u_{\alpha} \right) \right) \\
		&= \tau_t \tau_s u_t \tau_t. \left( u_t \tau_t. \left( u'. u_{\alpha} \right) \right) \\
		&= \tau_t \tau_s \tau_t. \left( u_t u'. u_{\alpha} \right)
	\end{align*}
	We distinguish the following cases:
	\begin{itemize}[leftmargin=*]
		\item $u = u_r$ for $r\in \{s, t\}$: Then $u' = 1$ and $u_r \tau_r. u_{\alpha} = u_{r\alpha} [u_{r\alpha}, u_r]$.
				
		\item $u = u_{st}$: Then $(\tau_s \tau_t)^3. \left( u_{ts}. u_{\alpha} \right) = (\tau_s \tau_t)^2 \tau_s. \left( u_s. u_{t\alpha} \right) = \tau_t. \left( u_s. u_{t\alpha} \right) = u_{ts}. u_{\alpha}$.
		
		\item $u = u_t u_s$: Then $u' = u_s$ and $u_t \tau_t. \left( u_s. u_{\alpha} \right) = u_t u_{ts}. u_{t\alpha} = u_{ts}. u_{t\alpha} [u_{t\alpha}, u_t]$.
		
		\item $u = u_t u_{st} u_s = u_s u_t$: Then the claim follows as in the previous case. 
		
		\item $u = u_{st} u_s$: Then $u' = u_{st}$ and we distinguish the following two cases:
		\begin{enumerate}[label=(\alph*)]
			\item $\{ \alpha_t, \alpha \} \in \mathcal{P}$: Note that $\{ s\alpha_t, s\alpha \} \in \mathcal{P}$ and $\alpha_s \subseteq s\alpha$. This implies $\{ s\alpha_t, \gamma \} \in \mathcal{P}$ for each $\gamma \in (\alpha_s, s\alpha)$. We compute $u_s \tau_s. \left( u_{st}. u_{\alpha} \right) = u_s u_t. u_{s\alpha} = u_t u_{st}. u_{s\alpha} [u_{s\alpha}, u_s] = u_t. u_{s\alpha} [u_{s\alpha}, u_{st}] [u_{s\alpha}, u_s] [[u_{s\alpha}, u_s], u_{st}]$ and the claim follows.
			
			\item $-\alpha_t \subseteq \alpha$: Then $\alpha_t \subseteq t\alpha$ and we deduce from the previous case that $(\tau_s \tau_t)^3. \left( u_{st} u_s. u_{\alpha} \right) = \tau_t (\tau_t \tau_s)^3. \left( u_s u_{st}. u_{t\alpha} \right) = \tau_t. \left( u_s u_{st}. u_{t\alpha} \right) = u_{st} u_s. u_{\alpha}$. This finishes the claim.
		\end{enumerate}
		
		\item $u = u_t u_{st}$: Interchanging $s$ and $t$ in the previous case, the claim follows. \qedhere
	\end{itemize}
\end{proof}

\begin{lemma}\label{Lemma: mst=4}
	Suppose $m_{st} = 4$. Then $[(\tau_s \tau_t)^4, n] = 1$ holds for all $n \in N_{s, t}$ in $P_s *_{U_+} P_t$.
\end{lemma}
\begin{proof}
	It suffices to show that $(\tau_s \tau_t)^2.n = (\tau_t \tau_s)^2.n$. If $n=u_{\alpha}$, then the claim follows. As before, we let $u = u_t u' \in U_{s, t}$ for some $u, u' \in U_{s, t}$ and assume that the braid relation acts trivial on $u'. u_{\alpha}$ and on $u_t \tau_t. \left( u'. u_{\alpha} \right)$. Then we compute the following:
	\allowdisplaybreaks
	\begin{align*}
		\tau_t \tau_s \tau_t \tau_s. \left( u_t u'. u_{\alpha} \right) &= \tau_t u_t \tau_s \tau_t \tau_s. \left( u'. u_{\alpha} \right) \\
		&= \tau_t u_t \tau_t \tau_s \tau_t \tau_s \tau_t. \left( u'. u_{\alpha} \right) \\
		&= u_t \tau_t u_t \tau_s \tau_t \tau_s \tau_t. \left( u'. u_{\alpha} \right) \\
		&= u_t \tau_t \tau_s \tau_t \tau_s. \left( u_t \tau_t. \left( u'. u_{\alpha} \right) \right) \\
		&= u_t \tau_s \tau_t \tau_s \tau_t. \left( u_t \tau_t. \left( u'. u_{\alpha} \right) \right) \\
		&= \tau_s \tau_t \tau_s u_t \tau_t. \left( u_t \tau_t. \left( u'. u_{\alpha} \right) \right) \\
		&= \tau_s \tau_t \tau_s \tau_t. \left( u_t u'. u_{\alpha} \right)
	\end{align*}
	We distinguish the following cases:
	\begin{itemize}[leftmargin=*]
		\item $u = u_r$ for some $r\in \{s, t\}$: Then $u' = 1$ and $u_r \tau_r. u_{\alpha} = u_{r\alpha} [u_{r\alpha}, u_r]$.
		
		\item A similar result holds for $u \in \{ u_{ts}, u_{st} \}$.
		
		\item $u = u_{ts} u_s$: Then $u' = u_{ts}$ and $u_s \tau_s. \left( u_{ts}. u_{\alpha} \right) = u_s u_{ts}. u_{s\alpha} = u_{ts}. u_{s\alpha} [u_{s\alpha}, u_s]$.
		
		\item $u = u_t u_{st}$: This follows similarly as in the previous case.
		
		\item $u = u_{st} u_s$: Then $u' = u_{st}$ and we distinguish the following two cases:
		\begin{enumerate}[label=(\alph*)]
			\item $\{ t\alpha_s, \alpha \} \in \P$: Note that $\{ t\alpha_s, s\alpha \} \in \mathcal{P}$ and $\alpha_s \subseteq s\alpha$. This implies $\{ t\alpha_s, \gamma \} \in \mathcal{P}$ for each $\gamma \in (\alpha_s, s\alpha)$. We have $u_s \tau_s. \left( u_{st}. u_{\alpha} \right) = u_s u_t. u_{s\alpha} = u_t u_{ts} u_{st}. u_{s\alpha} [u_{s\alpha}, u_s] = u_t u_{st}. u_{s\alpha} [u_{s\alpha}, u_{ts}] [u_{s\alpha}, u_s] [[u_{s\alpha}, u_s], u_{ts}]$. As we have already shown that the braid relation acts trivial in the cases $u' \in \{ 1, u_t, u_t u_{ts} \}$, the claim follows.
			
			\item $-t\alpha_s \subseteq \alpha$: Note that $t\alpha_s \subseteq tst\alpha$ and (as $-s\alpha_t \subseteq \alpha$) $\alpha_t \subseteq tst\alpha$. This implies $\{ t\alpha_s, \gamma \} \in \mathcal{P}$ for each $\gamma \in (\alpha_t, tst\alpha)$.
			
			We first check that the braid relation acts trivial on $u_t u_{ts}. u_{st\alpha}$. We have $u' = u_{ts}$ and we compute $u_t \tau_t. \left( u_{ts}. u_{st\alpha} \right) = u_t u_s. u_{tst\alpha} = u_s u_{st} u_{ts} u_t. u_{tst\alpha} = u_{st} u_s u_{ts}. u_{tst\alpha} [u_{tst\alpha}, u_t] = u_{st} u_s. u_{tst\alpha} [u_{tst\alpha}, u_{ts}] [u_{tst\alpha}, u_t] [[u_{tst\alpha}, u_t], u_{ts}]$. Now we can apply $(a)$ and the previous cases and the braid relation acts trivial on $u_t u_{ts}. u_{st\alpha}$. Thus we have $(\tau_s \tau_t)^4. \left( u_{st} u_s. u_{\alpha} \right) = \tau_t \tau_s (\tau_s \tau_t)^4. \left( u_t u_{ts}. u_{st\alpha} \right) = \tau_t \tau_s. \left( u_t u_{ts}. u_{st\alpha} \right) = u_{st} u_s. u_{\alpha}$ and the claim follows.
		\end{enumerate}
		
		\item $u \in \{ u_t u_{ts}, u_{ts} u_{st} \}$: Note that $\tau_t \tau_s. u_t u_{ts} = \tau_t. u_{st} u_{ts} = u_{st} u_s$.
		
		\item $u = u_t u_s$: Then $u' = u_s$ and $u_t \tau_t. \left( u_s. u_{\alpha} \right) = u_t u_{ts}. u_{t\alpha} = u_{ts}. u_{t\alpha} [u_{t\alpha}, u_t]$.
		
		\item $u = u_t u_{st} u_s$: Then $u' = u_{st} u_s$ and $u_t \tau_t. \left( u_{st} u_s. u_{\alpha} \right) = u_t u_{st} u_{ts}. u_{t\alpha} = u_{st} u_{ts}. u_{t\alpha} [u_{t\alpha}, u_t]$.
		
		\item $u = u_t u_{ts} u_s$: Then $u' = u_{ts} u_s$ and $u_t \tau_t. \left( u_{ts} u_s. u_{\alpha} \right) = u_t u_s u_{ts}. u_{t\alpha} = u_{st} u_s. u_{t\alpha} [u_{t\alpha}, u_t]$.
		
		\item $u = u_{ts} u_{st} u_s$: Then $u' = u_{ts} u_{st}$ and $u_s \tau_s. \left( u_{ts} u_{st}. u_{\alpha} \right) = u_s u_{ts} u_t. u_{s\alpha} = u_t u_{st}. u_{s\alpha} [u_{s\alpha}, u_s]$.
		
		\item $u = u_t u_{ts} u_{st}$: This follows similar as in the previous case.
		
		\item $u = u_t u_{ts} u_{st} u_s = u_s u_t$: Then $u' = u_t$ and $u_s \tau_s. \left( u_t. u_{\alpha} \right) = u_s u_{st}. u_{s\alpha} = u_{st}. u_{s\alpha} [u_{s\alpha}, u_s]$. \qedhere
	\end{itemize}
\end{proof}

\begin{lemma}\label{Lemma: mst=6}
	Suppose $m_{st} = 6$. Then $[(\tau_s \tau_t)^6, n] = 1$ holds for all $n \in N_{s, t}$ in $P_s *_{U_+} P_t$.
\end{lemma}
\begin{proof}
	Without loss of generality we assume that $(t, s) \in E(S)$ and we define $\tau_1 := \tau_s$ and $\tau_6 := \tau_t$. Moreover, we will write for short $1\alpha := \tau_1 \alpha$ ($6\alpha := \tau_6 \alpha$ respectively).
	
	It suffices to show that $(\tau_1 \tau_6)^3. n = (\tau_6 \tau_1)^3. n$. If $n = u_{\alpha}$, then the claim follows. As before, we let $u = u_i u' \in U_{s, t}$ for some $u, u' \in U_{s, t}$ and $i \in \{1, 6\}$ and assume that the braid relation acts trivial on $u'. u_{\alpha}$ and on $u_i \tau_i. \left( u'. u_{\alpha} \right)$. Let $j \in \{1, 6\}$ be different from $i$. Then we compute the following:
	\allowdisplaybreaks
	\begin{align*}
		\tau_i \tau_j \tau_i \tau_j \tau_i \tau_j. \left( u_i u'. u_{\alpha} \right) &= \tau_i u_i \tau_j \tau_i \tau_j \tau_i \tau_j. \left( u'. u_{\alpha} \right) \\
		&= \tau_i u_i \tau_i \tau_j \tau_i \tau_j \tau_i \tau_j \tau_i. \left( u'. u_{\alpha} \right) \\
		&= u_i \tau_i u_i \tau_j \tau_i \tau_j \tau_i \tau_j \tau_i. \left( u'. u_{\alpha} \right) \\
		&= u_i \tau_i \tau_j \tau_i \tau_j \tau_i \tau_j. \left( u_i \tau_i. \left( u'. u_{\alpha} \right) \right) \\
		&= u_i \tau_j \tau_i \tau_j \tau_i \tau_j \tau_i. \left( u_i \tau_i. \left( u'. u_{\alpha} \right) \right) \\
		&= \tau_j \tau_i \tau_j \tau_i \tau_j u_i \tau_i. \left( u_i \tau_i. \left( u'. u_{\alpha} \right) \right) \\
		&= \tau_j \tau_i \tau_j \tau_i \tau_j \tau_i. \left( u_i u'. u_{\alpha} \right)
	\end{align*}
	Before we distinguish all cases we compute a few commutation relations, which we will use without referring to them:
	\begin{enumerate}[label=(\arabic*)]
		\item $u_1 u_5 u_6 = u_1 u_6 [u_6, u_1] u_2 u_3 u_4 = u_6 u_1 [u_1, u_3] u_3 u_4 = u_6 u_4 u_3 u_1$;
		
		\item $u_1 u_3 u_5 = u_3 u_2 u_1 u_5 = u_3 u_2 u_5 u_4 u_2 u_1 = u_5 u_3 u_1$;
	\end{enumerate}
	\begin{itemize}[leftmargin=*]
		\item $u = u_i$ for some $i \in \{1, 6\}$: Then $u' = 1$ and $u_i \tau_i. u_{\alpha} = u_{i\alpha} [u_{i\alpha}, u_i]$.
		
		\item $u \in \{u_2, \ldots, u_5\}$: Then the claim follows similar as in the previous case.
		
		\item $u = u_2 u_1$: Then $u' = u_2$ and we distinguish the following cases:
		\begin{enumerate}[label=(\alph*)]
			\item $\{ \alpha_3, \alpha \} \in \mathcal{P}$: Then $\{ \alpha_5, 1\alpha \} \in \mathcal{P}$ and hence there exists $w\in (-\alpha_5) \cap (-1\alpha)$. Note that for all $i \in \{2, 3, 4\}$ we have $w\in (-\alpha_5) \cap (-1\alpha) \subseteq (-\alpha_5) \cap (-\alpha_1) \subseteq (-\alpha_i)$, as $\alpha_i \in (\alpha_1, \alpha_5)$. This implies $\alpha_1, \ldots, \alpha_5, 1\alpha \in \Phi(w) = \{ \alpha \in \Phi_+ \mid w \notin \alpha \}$. As $-\alpha_6 \subseteq 1\alpha$, we have $(-\alpha_6) \cap (-1\alpha) = \emptyset$ and hence $\alpha_6 \notin \Phi(w)$. Using Lemma \ref{Lemma: Definition of UG} there exists $\beta_1, \ldots, \beta_k \in \Phi(w) \backslash \{ \alpha_1, \ldots, \alpha_5 \}$ such that $u_5 \cdots u_1. u_{1\alpha} = u_{\beta_1} \cdots u_{\beta_k}$. We compute $u_1 \tau_1. \left( u_2. u_{\alpha} \right) = u_1 u_6. u_{1\alpha} = u_6 u_5 u_4 u_3 u_2 u_1. u_{1\alpha} = u_6. u_{\beta_1} \cdots u_{\beta_k}$ and the claim follows from the previous case.
			
			\item $\{ -\alpha_5, \alpha \}$ is prenilpotent: Then $\{ \alpha_3, 61616\alpha \} \in \mathcal{P}$ and $\alpha_6 \subseteq 61616\alpha$. As in $(a)$ we deduce $u_3 u_4 u_5 u_6. u_{61616\alpha} = u_{\beta_1} \cdots u_{\beta_m}$ for $\beta_i \notin \{ \alpha_1, \ldots, \alpha_6 \}$. As $w \notin \beta_i, \alpha_3$, we note that $\{ \alpha_3, \beta_i \} \in \mathcal{P}$. 
			
			We first show that the braid relation acts trivial on $u_6 u_5. u_{1616\alpha}$. We have $u' = u_5$ and we compute $u_6 \tau_6. \left( u_5. u_{1616\alpha} \right) = u_6 u_1. u_{61616\alpha} = u_1 u_2 u_3 u_4 u_5 u_6. u_{61616\alpha} = u_2 u_1. u_{\beta_1} \cdots u_{\beta_m}$. It follows from $(a)$ and the previous cases that that the braid relation acts trivial on $u_6 u_5. u_{1616\alpha}$ and we finally compute $(\tau_1 \tau_6)^2. \left( u_2 u_1. u_{\alpha} \right) = u_6 u_5. u_{1616\alpha}$.
			
			\item $\{ \alpha_4, \alpha \} \in \mathcal{P}$: As before, we have $u_4 u_3 u_2 u_1. u_{1\alpha} = u_{\beta_1} \cdots u_{\beta_k}$ with $\beta_i \notin \{ \alpha_1, \ldots, \alpha_6 \}$ and $\{ \alpha_3, \beta_i \} \in \mathcal{P}$. We compute $u_1 \tau_1. \left( u_2. u_{\alpha} \right) = u_1 u_6. u_{1\alpha} = u_6 u_5 u_4 u_3 u_2 u_1. u_{1\alpha} = u_6 u_5. u_{\beta_1} \cdots u_{\beta_k}$.
			
			We now have to show that the braid relation acts trivial on $u_6 u_5. u_{\beta_i}$ for every $1 \leq i \leq k$. As $\{ \alpha_3, \beta_i \} \in \mathcal{P}$, it follows that $\{-\alpha_5, 6161\beta_i\}$ is a prenilpotent pair of roots. We can now apply $(b)$ to deduce that the braid relation acts trivial on $u_6 u_5. u_{\beta_1} \cdots u_{\beta_k}$.
		
			\item $-\alpha_4, \alpha_5 \subseteq \alpha$: Then $\alpha_2, (-\alpha_3) \subseteq 1616\alpha$. Note that $(\tau_1 \tau_6)^2. \left( u_2 u_1. u_{\alpha} \right) = u_6 u_5. u_{1616\alpha}$ and it suffices to show that the braid relation acts trivial on $u_6 u_5. u_{1616\alpha}$.
						
			We have $u' = u_5$ and it suffices to show that the braid relation acts trivial on $u_6 \tau_6. (u_5. u_{1616\alpha})$. As $-\alpha_4, \alpha_5 \subseteq \alpha$, we deduce $\alpha_4, -\alpha_3 \subseteq 61616\alpha$. This implies $\{ \alpha_4, 61616\alpha \}, \{ \alpha_6, 61616\alpha \} \in \mathcal{P}$. As before, we have $u_4 u_5 u_6. u_{61616\alpha} = u_{\beta_1} \cdots u_{\beta_k}$ with $\{ -\alpha_3, \beta_i \}$ prenilpotent and $\{ \alpha_4, \beta_i \} \in \mathcal{P}$. We compute $u_6 \tau_6. \left( u_5. u_{1616\alpha} \right) = u_6 u_1. u_{61616\alpha} = u_1 u_2 u_3 u_4 u_5 u_6. u_{61616\alpha} = u_1 u_2 u_3. u_{\beta_1} \cdots u_{\beta_l}$. We have to show that the braid relation acts trivial on $u_1 u_2 u_3. u_{\beta_i}$ for every $1 \leq i \leq k$. If $\{ \alpha_1, \beta_i \} \in \mathcal{P}$, we have $u_1 u_2 u_3. u_{\beta_i} = u_3. u_{\beta_i} [u_{\beta_i}, u_1]$ and the claim follows from the previous cases. Thus we can assume that $-\alpha_1 \subseteq \beta_i$.
			
			Note that $\{ \alpha_5, \beta_i \} \in \mathcal{P}$ and hence $\{ -\alpha_5, 616\beta_i \}$ is a prenilpotent pair of roots. It follows from $(b)$ that the braid relation acts trivial on $u_2 u_1. u_{616\beta_i}$. As $\tau_6 \tau_1 \tau_6. \left( u_2 u_3. u_{\beta_i} \right) = u_2 u_1. u_{616\beta_i} $, the braid relation also acts trivial on $u_2 u_3. u_{\beta_i}$. We compute $u_1 \tau_1. \left( u_2 u_3. u_{\beta_i} \right) = u_1 u_6 u_5. u_{1\beta_i} = u_6 u_4 u_3 u_1. u_{1\beta_i}$. As $\alpha_1 \subseteq 1\beta_i$ and $\{ \alpha_4, 1\beta_i \} \in \mathcal{P}$, we have $u_4 u_3 u_1. u_{1\beta_i} = u_{\gamma_1} \cdots u_{\gamma_m}$ as before and the claim follows.
		\end{enumerate}
		
		\item $u \in \{ u_3 u_2, u_4 u_3, u_5 u_4, u_6 u_5 \}$: Note that $(\tau_1 \tau_6)^2. u_2 u_1 = \tau_1 \tau_6 \tau_1. u_4 u_5 = \tau_1 \tau_6. u_4 u_3 = \tau_1. u_2 u_3 = u_6 u_5$.
		
		\item $u = u_6 u_4$: Then $u' = u_4$ and we distinguish the following two cases:
		\begin{enumerate}[label=(\alph*)]
			\item $\{ \alpha_2, \alpha \} \in \P$: Note that $\{ \alpha_4, 6\alpha \} \in \P$ and $\alpha_6 \subseteq 6\alpha$. This implies $\{ \alpha_4, \gamma \} \in \mathcal{P}$ for each $\gamma \in (\alpha_6, 6\alpha)$. We compute $u_6 \tau_6. \left( u_4. u_{\alpha} \right) = u_6 u_2. u_{6\alpha} = u_2 [u_2, u_6]. u_{6\alpha} [u_{6\alpha}, u_6] = u_2. u_{6\alpha} [u_{6\alpha}, u_4] [u_{6\alpha}, u_6] [[u_{6\alpha}, u_6], u_4]$.
			
			\item $-\alpha_2 \subseteq \alpha$: Then $\alpha_2 \subseteq 161\alpha$ and we deduce from the previous case $(\tau_6 \tau_1)^6. \left( u_6 u_4. u_{\alpha} \right) = \tau_1 \tau_6 \tau_1 (\tau_1 \tau_6)^6. \left( u_4 u_6. u_{161\alpha} \right) = \tau_1 \tau_6 \tau_1. \left( u_4 u_6. u_{161\alpha} \right) = u_6 u_4 u_{\alpha} u_4 u_6$.
		\end{enumerate}
		
		\item $u = u_4 u_2$: Note $\tau_1. u_4 u_2 = u_4 u_6 = u_6 u_4$.
		
		\item $u = u_6 u_2$: Then $u' = u_2$ and $u_6 \tau_6. \left( u_2. u_{\alpha} \right) = u_6 u_4. u_{6\alpha} = u_4. u_{6\alpha} [u_{6\alpha}, u_6]$.
		
		\item $u = u_6 u_4 u_2$: Note $\tau_1. u_6 u_4 u_2 = u_2 u_4 u_6 = u_6 u_2$.
		
		\item $u = u_3 u_1 = u_1 u_3 u_2$: Then $u' = u_3 u_2$ and we distinguish the following cases:
		\begin{enumerate}[label=(\alph*)]
			\item $\{ \alpha_5, \alpha \} \in \mathcal{P}$: Note that $\{ \alpha_3, 1\alpha \} \in \mathcal{P}$ and $\alpha_1 \subseteq 1\alpha$. This implies $\{ \alpha_3, \gamma \} \in \mathcal{P}$ for each $\gamma \in (\alpha_1, 1\alpha)$. We compute $u_1 \tau_1. \left( u_3 u_2. u_{\alpha} \right) = u_1 u_5 u_6. u_{1\alpha} = u_6 u_4 u_3. u_{1\alpha} [u_{1\alpha}, u_1] = u_6 u_4. u_{1\alpha} [u_{1\alpha}, u_3] [u_{1\alpha}, u_1] [[u_{1\alpha}, u_1], u_3]$.
		
			\item $-\alpha_5 \subseteq \alpha$: Then $\alpha_5 \subseteq 616\alpha$ and we first show that in this case the braid relation acts trivial on $u_1 u_3. u_{616\alpha}$. Note that $\alpha_1, \alpha_3 \subseteq 1616\alpha$ and, as $\alpha_2 \in (\alpha_1, \alpha_3)$, we have $(-\alpha_1) \cap (-\alpha_3) \subseteq (-\alpha_2)$ and hence $\alpha_2 \subseteq \alpha_1 \cup \alpha_3 \subseteq 1616\alpha$. This implies $\{ \alpha_2, \gamma \} \in \mathcal{P}$ for each $\gamma \in (\alpha_1, 1616\alpha)$. We have $u' = u_3$ and $u_1 \tau_1. \left( u_3. u_{616\alpha} \right) = u_1 u_5. u_{1616\alpha} = u_5 u_4 u_2. u_{1616\alpha} [u_{1616\alpha}, u_1] = u_5 u_4. u_{1616\alpha} [u_{1616\alpha}, u_2] [u_{1616\alpha}, u_1] [[u_{1616\alpha}, u_1], u_2]$. Using $(a)$ and the previous cases we deduce that the braid relation acts trivial on $u_1 u_3. u_{616\alpha}$. We conclude $(\tau_1 \tau_6)^6. \left( u_3 u_1. u_{\alpha} \right) = \tau_6 \tau_1 \tau_6 (\tau_6 \tau_1)^6. \left( u_1 u_3. u_{616\alpha} \right) = \tau_6 \tau_1 \tau_6. \left( u_1 u_3. u_{616\alpha} \right) = u_3 u_1. u_{\alpha}$.
		\end{enumerate}
	
		\item $u\in \{ u_3 u_2 u_1, u_5 u_3, u_5 u_4 u_3 \}$: Then $\tau_6 \tau_1 \tau_6. u_3 u_2 u_1 = \tau_6 \tau_1. u_3 u_4 u_5 = \tau_6. u_5 u_4 u_3 = u_1 u_2 u_3$ and $u_3 u_4 u_5 = u_5 u_3, u_1 u_2 u_3 = u_3 u_1$.
		
		\item $u = u_4 u_1$: Then $u' = u_4$ and $u_1 \tau_1. \left( u_4. u_{\alpha} \right) = u_1 u_4. u_{1\alpha} = u_4. u_{1\alpha} [u_{1\alpha}, u_1]$.
		
		\item $u \in \{ u_5 u_2, u_6 u_3 \}$: Note that $\tau_6 \tau_1. u_6 u_3 = \tau_6. u_2 u_5 = u_4 u_1$.
		
		\item $u = u_5 u_4 u_2 u_1 = u_1 u_5$: Then $u' = u_5$ and $u_1 \tau_1. \left( u_5. u_{\alpha} \right) = u_1 u_3. u_{1\alpha} = u_3 u_2. u_{1\alpha} [u_{1\alpha}, u_1]$.
		
		\item $u = u_5 u_1$: Note $\tau_6. u_5 u_1 = u_1 u_5 = u_5 u_4 u_2 u_1$.
		
		\item $u = u_6 u_1$: Then $u' = u_1$ and $u_6 \tau_6. \left( u_1. u_{\alpha} \right) = u_6 u_5. u_{6\alpha} = u_5. u_{6\alpha} [u_{6\alpha}, u_6]$.
		
		\item $u = u_5 u_3 u_1 = u_1 u_3 u_5 = u_1 u_5 u_4 u_3$: Then $u' = u_5 u_4 u_3$ and $u_1 \tau_1. \left( u_5 u_4 u_3. u_{\alpha} \right) = u_1 u_3 u_4 u_5. u_{1\alpha} = u_2 u_1 [u_1, u_6]. u_{1\alpha} = u_2 [u_6, u_1] u_1. u_{1\alpha} = u_5 u_4 u_3. u_{1\alpha} [u_{1\alpha}, u_1]$.
		
		\item $u = u_6 u_5 u_4$: Then $u' = u_5 u_4$ and $u_6 \tau_6. \left( u_5 u_4. u_{\alpha} \right) = u_6 u_1 u_2. u_{6\alpha} = u_2 [u_2, u_6] u_6 u_1. u_{6\alpha} = u_2 u_4 [u_6, u_1] u_1 u_6. u_{6\alpha} = u_5 u_3 u_1. u_{6\alpha} [u_{6\alpha}, u_6]$.
		
		\item $u = u_4 u_3 u_2$: Note that $\tau_1. u_4 u_3 u_2 = u_4 u_5 u_6 = u_6 u_5 u_4$.
		
		\item $u = u_4 u_3 u_1 = u_1 u_4 u_3 u_2$: Then $u' = u_4 u_3 u_2$ and $u_1 \tau_1. \left( u_4 u_3 u_2. u_{\alpha} \right) = u_1 u_4 u_5 u_6. u_{1\alpha} = u_4. \left( u_1 u_5 u_6. u_{1\alpha} \right) = u_4. \left( u_6 u_4 u_3 u_1. u_{1\alpha} \right) = u_6 u_3. u_{1\alpha} [u_{1\alpha}, u_1]$.
	
		\item $u \in \{ u_6 u_5 u_3, u_5 u_4 u_3 u_2 \}$: Note $\tau_6 \tau_1. u_6 u_5 u_3 = \tau_6. u_2 u_3 u_5 = u_4 u_3 u_1$ and $u_2 u_3 u_5 = u_5 u_4 u_3 u_2$.
		
		\item $u = u_6 u_4 u_3$: Then $u' = u_4 u_3$ and $u_6 \tau_6. \left( u_4 u_3. u_{\alpha} \right) = u_6 u_2 u_3. u_{6\alpha} = u_2 [u_2, u_6] u_6 u_3. u_{6\alpha} = u_4 u_3 u_2. u_{6\alpha} [u_{6\alpha}, u_6]$.
	
		\item $u \in \{ u_4 u_2 u_1, u_5 u_4 u_2 \}$: Note that $\tau_1 \tau_6. u_4 u_2 u_1 = \tau_1. u_2 u_4 u_5 = u_6 u_4 u_3$ and $u_2 u_4 u_5 = u_5 u_4 u_2$.
		
		\item $u = u_5 u_2 u_1 = u_1 u_5 u_4$: Then $u' = u_5 u_4$ and $u_1 \tau_1. \left( u_5 u_4. u_{\alpha} \right) = u_1 u_3 u_4. u_{1\alpha} = u_4 u_3 u_2. u_{1\alpha} [u_{1\alpha}, u_1]$.
		
		\item $u = u_6 u_2 u_1$: Then $u' = u_2 u_1$ and $u_6 \tau_6. \left( u_2 u_1. u_{\alpha} \right) = u_6 u_4 u_5. u_{6\alpha} = u_5 u_4. u_{6\alpha} [u_{6\alpha}, u_6]$.
		
		\item $u = u_6 u_3 u_1$: Then $u' = u_3 u_1$ and $u_6 \tau_6.\left( u_3 u_1. u_{\alpha} \right) = u_6 u_3 u_5. u_{6\alpha} = u_5 u_4 u_3. u_{6\alpha} [u_{6\alpha}, u_6]$.
	
		\item $u = u_5 u_4 u_1 = u_1 u_5 u_2$: Then $u' = u_5 u_2$ and we have $u_1 \tau_1. \left( u_5 u_2. u_{\alpha} \right) = u_1 u_3 u_6. u_{1\alpha} = u_3 [u_3, u_1] [u_1, u_6] u_6 u_1. u_{1\alpha} = u_4 u_5 u_6 u_1. u_{1\alpha} = u_6 u_5 u_4. u_{1\alpha} [u_{1\alpha}, u_1]$.
	
		\item $u = u_6 u_4 u_1$: Then $u' = u_4 u_1$ and $u_6 \tau_6. \left( u_4 u_1. u_{\alpha} \right) = u_6 u_2 u_5. u_{6\alpha} = u_2 [u_2, u_6] u_6 u_5. u_{6\alpha} = u_5 u_4 u_2. u_{6\alpha} [u_{6\alpha}, u_6]$.
	
		\item $u = u_6 u_5 u_1 = u_6 u_5 u_4 u_3 [u_3, u_1] u_1 u_3 u_4 = u_6 [u_6, u_1] u_1 u_3 u_4 = u_1 u_6 u_4 u_3$: Then $u' = u_6 u_4 u_3$ and $u_1 \tau_1. \left( u_6 u_4 u_3. u_{\alpha} \right) = u_1 u_2 u_4 u_5. u_{1\alpha} = u_1 [u_1, u_5] u_5. u_{1\alpha} = u_5. u_{1\alpha} [u_{1\alpha}, u_1]$.
	
		\item $u = u_6 u_3 u_2$: Then $u' = u_3 u_2$ and $u_6 \tau_6. \left( u_3 u_2. u_{\alpha} \right) = u_6 u_3 u_4. u_{6\alpha} = u_4 u_3. u_{6\alpha} [u_{6\alpha}, u_6]$.
		
		\item $u = u_6 u_5 u_4 u_2 = u_2 u_5 u_6$: Then $\tau_1. u_2 u_5 u_6 = u_6 u_3 u_2$.
		
		\item $u = u_5 u_4 u_3 u_1 = u_1 u_3 u_5 u_4 = u_1 u_5 u_3$: Then $u' = u_5 u_3$ and $u_1 \tau_1. \left( u_5 u_3. u_{\alpha} \right) = u_1 u_3 u_5. u_{1\alpha} = u_3 [u_3, u_1] [u_1, u_5] u_5 u_1. u_{1\alpha} = u_3 u_2 u_2 u_4 u_5 u_1. u_{1\alpha} = u_5 u_3. u_{1\alpha} [u_{1\alpha}, u_1]$.
		
		\item $u = u_5 u_3 u_2 u_1$: Then $\tau_6. u_5 u_3 u_2 u_1 = u_1 u_3 u_4 u_5 = u_5 u_4 u_3 u_1$.
	
		\item $u = u_4 u_3 u_2 u_1 = u_1 u_4 u_3$: Then $u' = u_4 u_3$ and $u_1 \tau_1. \left( u_4 u_3. u_{\alpha} \right) = u_1 u_4 u_5. u_{1\alpha} = u_4 u_5 [u_5, u_1] u_1. u_{1\alpha} = u_5 u_2. u_{1\alpha} [u_{1\alpha}, u_1]$.
			
		\item $u \in \{ u_6 u_5 u_4 u_3, u_5 u_3 u_2 \}$: Note $\tau_6 \tau_1. u_6 u_5 u_4 u_3 = \tau_6. u_2 u_3 u_4 u_5 = u_4 u_3 u_2 u_1$ and $u_2 u_3 u_4 u_5 =  u_5 u_3 u_2$.
		
		\item $u = u_6 u_5 u_2$: Then $u' = u_5 u_2$ and $u_6 \tau_6. \left( u_5 u_2. u_{\alpha} \right) = u_6 u_1 u_4. u_{6\alpha} = u_4 [u_6, u_1] u_1 u_6. u_{6\alpha} = u_5 u_3 u_2 u_1. u_{6\alpha} [u_{6\alpha}, u_6]$.
	
		\item $u = u_6 u_4 u_3 u_2$: Then $\tau_1. u_6 u_4 u_3 u_2 = u_2 u_4 u_5 u_6 = u_6 [u_6, u_2] u_2 u_4 u_5 = u_6 u_5 u_2$.
		
		\item $u = u_6 u_3 u_2 u_1$: Then $u' = u_3 u_2 u_1$ and $u_6 \tau_6. \left( u_3 u_2 u_1. u_{\alpha} \right) = u_6 u_3 u_4 u_5. u_{6\alpha} = u_5 u_3. u_{6\alpha} [u_{6\alpha}, u_6]$.
	
		\item $u = u_6 u_4 u_2 u_1$: Then $u' = u_4 u_2 u_1$ and $u_6 \tau_6. \left( u_4 u_2 u_1. u_{\alpha} \right) = u_6 u_2 u_4 u_5. u_{6\alpha} = u_6 u_2 [u_2, u_6] u_5. u_{6\alpha} = u_5 u_2. u_{6\alpha} [u_{6\alpha}, u_6]$.
	
		\item $u = u_6 u_5 u_2 u_1$: Then $u' = u_5 u_2 u_1$ and $u_6 \tau_6. \left( u_5 u_2 u_1. u_{\alpha} \right) = u_6 u_1 u_4 u_5. u_{6\alpha} = u_1 [u_1, u_6] u_6 u_4 u_5. u_{6\alpha} = u_1 u_2 u_3 u_6. u_{6\alpha} = u_3 u_1. u_{6\alpha} [u_{6\alpha}, u_6]$.
		
		\item $u = u_6 u_4 u_3 u_1 = u_5 u_6 [u_6, u_1] u_1 u_2 = u_5 u_1 u_6 u_2 = u_1 u_6 u_5$: Then $u' = u_6 u_5$ and $u_1 \tau_1. \left( u_6 u_5. u_{\alpha} \right) = u_1 u_2 u_3. u_{1\alpha} = u_1 [u_1, u_3] u_3. u_{1\alpha} = u_3. u_{1\alpha} [u_{1\alpha}, u_1]$.
			
		\item $u = u_6 u_5 u_3 u_1 =  u_6 [u_6, u_1] u_1 u_4 u_2 = u_1 u_6 u_4 u_2$: Then $u' = u_6 u_4 u_2$ and $u_1 \tau_1. \left( u_6 u_4 u_2. u_{\alpha} \right) = u_1 u_2 u_4 u_6. u_{1\alpha} = u_6 [u_6, u_1] u_1 u_2. u_{1\alpha} = u_6 u_5 u_4 u_3. u_{1\alpha} [u_{1\alpha}, u_1]$.
	
		\item $u = u_6 u_5 u_4 u_1 = u_6 [u_6, u_1] u_1 u_3 = u_1 u_6 u_3$: Then $u' = u_6 u_3$ and $u_1 \tau_1. \left( u_6 u_3. u_{\alpha} \right) = u_1 u_2 u_5. u_{1\alpha} = u_5 [u_5, u_1] u_1 u_2. u_{1\alpha} = u_5 u_4. u_{1\alpha} [u_{1\alpha}, u_1]$.
		
		\item $u = u_6 u_5 u_3 u_2$: Then $u' = u_5 u_3 u_2$ and $u_6 \tau_6. \left( u_5 u_3 u_2. u_{\alpha} \right) = u_6 u_1 u_3 u_4. u_{6\alpha} = u_5 u_6 u_5 u_4 u_3 u_2 u_1. u_{6\alpha} = u_5 u_6 [u_6, u_1] u_1. u_{6\alpha} = u_5 u_1. u_{6\alpha} [u_{6\alpha}, u_6]$.
		
		\item $u = u_5 \cdots u_1 = [u_6, u_1] u_1 = u_1 [u_1, u_6] = u_1 u_2 u_3 u_4 u_5 = u_1 u_5 u_3 u_2$: Then $u' = u_5 u_3 u_2 = u_2 u_3 u_4 u_5$ and $u_1 \tau_1. \left( u_2 u_3 u_4 u_5. u_{\alpha} \right) = u_1 u_6 u_5 u_4 u_3. u_{1\alpha} = u_1 u_6 [u_6, u_1] u_2. u_{1\alpha} = u_6 u_2. u_{1\alpha} [u_{1\alpha}, u_1]$.
		
		\item $u = u_6 u_4 u_3 u_2 u_1$: Then $u' = u_4 u_3 u_2 u_1$ and $u_6 \tau_6. \left( u_4 u_3 u_2 u_1. u_{\alpha} \right) = u_6 u_2 u_3 u_4 u_5. u_{6\alpha} = u_6 [u_1, u_6]. u_{6\alpha} = [u_6, u_1] u_6. u_{6\alpha} = u_5 u_4 u_3 u_2. u_{6\alpha} [u_{6\alpha}, u_6]$.
		
		\item $u = u_6 u_5 u_3 u_2 u_1 = u_6 [u_6, u_1] u_1 u_4 = u_1 u_6 u_4$: Then $u' = u_6 u_4$ and $u_1 \tau_1. \left( u_6 u_4. u_{\alpha} \right) = u_1 u_2 u_4. u_{1\alpha} = u_4 u_2. u_{1\alpha} [u_{1\alpha}, u_1]$.
		
		\item $u = u_6 u_5 u_4 u_2 u_1 = u_6 [u_6, u_1] u_1 u_3 u_2 = u_1 u_6 u_3 u_2$: Then $u' = u_6 u_3 u_2$ and $u_1 \tau_1. \left( u_6 u_3 u_2. u_{\alpha} \right) = u_1 u_2 u_5 u_6. u_{1\alpha} = u_2 u_5 [u_5, u_1] u_1 u_6. u_{1\alpha} = u_4 u_5 u_6 [u_6, u_1] u_1. u_{1\alpha} = u_6 u_3 u_2. u_{1\alpha} [u_{1\alpha}, u_1]$.
		
		\item $u = u_6 u_5 u_4 u_3 u_1 = u_6 [u_6, u_1] u_1 u_2 = u_1 u_6 u_2$: Then $u' = u_6 u_2$ and $u_1 \tau_1. \left( u_6 u_2. u_{\alpha} \right) = u_1 u_2 u_6. u_{1\alpha} = u_1 u_6 u_4 u_2. u_{1\alpha} = u_6 [u_6, u_1] u_4 u_2 u_1. u_{1\alpha}= u_6 u_5 u_3. u_{1\alpha} [u_{1\alpha}, u_1]$.
	
		\item $u = u_6 u_5 u_4 u_3 u_2$: Then $u' = u_5 u_4 u_3 u_2$ and $u_6 \tau_6.\left( u_5 u_4 u_3 u_2.u_{\alpha} \right) = u_6 u_1 u_2 u_3 u_4. u_{6\alpha} = u_6 u_1 [u_1, u_6] u_5. u_{6\alpha} = u_1 u_6 u_5. u_{6\alpha} = u_5 u_4 u_2 u_1. u_{6\alpha} [u_{6\alpha}, u_6]$.
		
		\item $u = u_6 \cdots u_1 = u_6 [u_6, u_1] u_1 = u_1 u_6$: Then $u' = u_6$ and $u_1 \tau_1. \left( u_6. u_{\alpha} \right) = u_1 u_2. u_{1\alpha} = u_2. u_{1\alpha} [u_{1\alpha}, u_1]$. \qedhere
	\end{itemize}
\end{proof}

\bibliographystyle{amsalpha}
\bibliography{references}

\end{document}